\documentclass[12pt,a4paper]{article}

\usepackage[nottoc]{tocbibind}
\usepackage{preamble}
\newcommand{\usualHilb}{{\ccH}ilb}
\DeclareMathOperator{\sat}{sat}

 \usepackage{vmargin}
 \setmargrb{1in}{0.7in}{1in}{0.7in} 

\title{Apolarity, border rank and multigraded Hilbert scheme}
\author{\nisiabu \and \JaBu}

\date{17th October 2020}

\newenvironment{red}{\color{red}}{}
\newcommand{\bred}{\begin{red}}
\newcommand{\ered}{\end{red}}

\newenvironment{blue}{\color{blue}}{}
\newcommand{\bblue}{\begin{blue}}
\newcommand{\eblue}{\end{blue}}

\newenvironment{green}{\color{green}}{}
\newcommand{\bgreen}{\begin{green}}
\newcommand{\egreen}{\end{green}}

\usepackage[textsize=tiny]{todonotes}

\newcommand{\bVSP}{\underline{\mathbf{VSP}}}
\newcommand{\Irrel}[1]{\operatorname {Irrel}_{#1}}
\newcommand{\rank}[2][X]{\operatorname{r}_{#1}(#2)}
\newcommand{\borderrank}[2][X]{\operatorname{br}_{#1}\!\left(#2\right)}

\newcommand{\apolar}[1][F]{\operatorname{Ann}\!\left(#1\right)}

\DeclareMathOperator{\Slip}{Slip}
\newcommand{\sH}{\operatorname{H}}
\DeclareMathOperator{\Sip}{Sip}

\DeclareMathOperator{\idealof}{I}
\DeclareMathOperator{\zeroscheme}{Z}
\DeclareMathOperator{\Lex}{Lex}

\newcommand{\grevlex}{\texttt{grevlex}}

\definecolor{DziubkaDyniowy}{RGB}{224,104,19}
\definecolor{geometria}{RGB}{170,67,6} 
\definecolor{DziubkaBrazA}{RGB}{140,64,3}
\definecolor{algebra}{RGB}{75,135,10} 

\begin{document}

\maketitle
\begin{abstract}
We introduce an elementary method to study the border rank
of polynomials and tensors, analogous to the apolarity lemma.
This can be used to describe the border rank of all cases uniformly, including those very special ones that resisted a systematic approach.
We also define a border rank version of the variety of sums of powers and analyse its usefulness in studying tensors and polynomials with large symmetries.
In particular, it can be applied to provide lower bounds for the border rank of some  interesting tensors, such as the matrix multiplication tensor.
We work in a general setting, where the base variety is not necessarily a Segre or Veronese variety, but an arbitrary smooth toric projective variety.
A critical ingredient of our work is an irreducible component of a multigraded Hilbert scheme related to the toric variety in question.
\end{abstract}


\medskip
{\footnotesize
\noindent\textbf{addresses:} \\
W.~Buczy\'nska, \eemail{wkrych@mimuw.edu.pl}, 
   Faculty of Mathematics, Computer Science and Mechanics, University of Warsaw, ul. Banacha 2, 02-097 Warsaw, Poland\\
J.~Buczy\'nski, \eemail{jabu@mimuw.edu.pl}, 
   Institute of Mathematics of the Polish Academy of Sciences, ul. \'Sniadeckich 8, 00-656 Warsaw, Poland

\medskip

\noindent\textbf{keywords:} 
Border rank, apolarity, multigraded Hilbert scheme, variety of sums of powers, invariant ideals, Cox ring.

}

\tableofcontents
\section{Introduction}

The majority of the techniques used in this article 
  originate in algebraic geometry.
This introduction is meant to be accessible also to researchers in applications.  
  
Tensor rank, Waring rank or their common generalisation called partially symmetric rank are of interest to mathematicians due to their many applications to computational complexity, quantum physics, and algebraic statistics, but also due to their geometric and algebraic interpretations. 
Explicitly, we consider one of the classical projective varieties $X\subset \PP^N$: the Segre variety (for tensor rank) consisting of simple tensors, the Veronese variety (for Waring rank) consisting of powers of linear forms, or the Segre-Veronese variety (for partially symmetric rank), which is a combination of the two above.
In all these cases, abstractly, 
$X \simeq \PP^{a} \times \PP^{b}\times \PP^c \times
   \dotsb$, 
   that is, $X$ is a finite product of projective spaces.
The embedding $X\subset \PP^N$ depends on the symmetries 
   of the tensor and determines a sequence of degrees 
   $L= (l_1, l_2, l_3, \dotsc)$ of the same length as the number of factors of $X$.
Then the rank of $F$, which is, respectively, a tensor, a homogeneous polynomial, or a partially symmetric tensor, is defined as the minimal integer $r$ such that $F$ is in the linear span of $r$ distinct points of $X$.

One of the methods to obtain values of rank 
  for specific tensors or polynomials that has been shown 
  to be very useful is the method of apolarity.
It exploits a multigraded polynomial ring 
  \[
    S[X]= \CC[\fromto{\alpha_0, \alpha_1}{\alpha_a},
              \fromto{\beta_0}{\beta_b},
              \fromto{\gamma_0}{\gamma_c},\dotsc
              ], 
  \]
  which is graded by as many copies of $\ZZ$ 
  as there are factors of the projective space in $X$:
  all $\alpha_i$'s have multidegree $(1,0,0,\dotsc)$,
  all $\beta_j$'s have multidegree $(0,1,0,\dotsc)$,
  and so on.
  (In the Veronese case, that is, $X\simeq \PP^a$, 
    $S[X]$ is just the standard homogeneous coordinate ring 
    of $\PP^a$.)
  The ring $S[X]$ has two dual interpretations that are illustrated
    in diagram \eqref{equ_duality_of_apolarity} below.
  The first interpretation is more geometric: 
    the elements of $S[X]$ are thought of as ``functions''
    on $X$, which makes $S[X]$ into a kind of a multigraded coordinate ring of $X$.
  Strictly speaking, those functions are sections of line bundles on $X$,
    and 
    \begin{align*}
      S[X] &= \bigoplus_{D\in \Pic(X)} H^0(X,D) = \bigoplus_{D\in \Pic(X)} S[X]_D \\
           &= \bigoplus_{(d_1, d_2,d_3,\dotsc)} H^0\left(\PP^a\times \PP^b \times \PP^c\times \dotsc,\ccO_{\PP^a\times \PP^b \times \PP^c\times \dotsc}(d_1, d_2,d_3,\dotsc)\right).
    \end{align*}
  Here $D=(d_1, d_2, d_3, \dotsc)\in \Pic(X)$ and $S[X]_D$ is the $D$-th graded piece of $S[X]$, that is, the space of polynomials of multidegree $(d_1, d_2, d_3, \dotsc)$.
  The second, more algebraic, interpretation of $S[X]$
     is in terms of derivations.
  In this interpretation
    the variables of $S[X]$ can be seen as derivations 
    of $F \in \PP^N= \PPof{S[X]_{(l_1, l_2, l_3,\dotsc)}}^*$. 
  Then $F$ determines a multihomogeneous ideal $\apolar\subset S[X]$, which is responsible for many algebraic properties 
    of $F$ and other objects constructed from $F$ 
(see for instance~\cite{iarrobino_kanev_book_Gorenstein_algebras}).
Explicitly, 
$\apolar$ is the set of all polynomial differential operators 
  with constant coefficients that annihilate $F$.
The symbol $\hook$ is used to denote the action of  derivations on tensors.
See Section~\ref{sec_apolarity_for_rank} for a more formal definition 
  and its reinterpretations.

  The very essence of this article is to expand the  apolarity theory, which in its simplest incarnation 
  can be seen as the following 
  duality between geometry and algebra.
\begin{equation}\label{equ_duality_of_apolarity}
\arraycolsep=1pt\def\arraystretch{1.4}
\begin{array}[b]{rlcrl}
    & \textbf{\color{geometria}Geometry} &&\textbf{\color{algebra}Algebra}\\
    \color{geometria}\setfromto{p_1}{p_r} &\color{geometria}\subset X \subset \PP\left(H^0(L)^*\right) & =&
     \PP{\color{algebra}\left(S[X]_L\right)^*}\ni & F\\
    \idealof(\setfromto{p_1}{p_r}) &\subset  \textstyle{\bigoplus_{D\in \Pic(X)}}  \color{geometria}H^0(D) & = & \color{algebra} S[X]   \supset& \color{algebra}\apolar \\
    \parallel \qquad \ &&&&\qquad \parallel \\
    \bigl(\Theta \mid \Theta \in H^0(D),& \forall_i  \ \Theta(p_i) =0 \bigr)   &&\color{algebra}\bigl(\Theta \mid&\color{algebra}\Theta \in S[X]_D , \Theta\hook F = 0 \bigr) \\
    \\
\quad\quad F &\in \color{geometria}\langle\{p_1,\dotsc, p_r\}\rangle &\Longleftrightarrow&\quad \idealof(\{ p_1,\dotsc, p_r\})&\subset \color{algebra}\apolar . 
\end{array}
\end{equation}
The left hand side of the diagram (in particular, the {\color{geometria}brown coloured} bits) represents geometric objects such 
  as the Segre-Veronese variety, projective space, points, line bundles and their sections, and linear span.
The right hand side (especially, the stuff {\color{algebra}in green}) contains algebraic objects: 
  the polynomial ring $S[X]$,
  the apolar ideal of differentials annihilating~$F$.
The equivalence in the bottom line is called the Apolarity Lemma 
  (see Proposition~\ref{prop_multigraded_apolarity}).

Despite the equivalence of the Apolarity Lemma, it is in general 
  difficult to 
  obtain the exact value of rank for many explicit tensors, 
  or to describe explicitly the stratification of the projective space 
  $\PP^N$ by rank.
One of the reasons is that the rank is not semicontinuous, as many standard examples show.
Thus often more natural for calculations, and also applications, are the notions of secant variety and border rank: 
\begin{itemize}
 \item the $r$-th \emph{secant variety} of $X \subset \PP^N$ is
 $
   \sigma_r(X) = \overline{\bigcup \set{ \linspan{\fromto{p_1}{p_r}} 
     \mid p_i\in X}} \subset \PP^N,
 $
 \item the \emph{border rank} of $F$ is 
$\borderrank{F} = \min \set{r \in \ZZ \mid F \in \sigma_r(X)}$.
\end{itemize}
The border rank is lower semicontinuous and the secant varieties are algebraic subsets of $\PP^N$.
One of the major problems 
  is to estimate the growth of the border rank of tensors representing 
  matrix multiplication of large matrices
 \cite{landsberg_geometry_and_complexity}.
  
One of the missing pieces in this area was the analogue of apolarity theory for border rank, and it is the topic of this article.
\begin{thm}[Weak border apolarity]
    \label{thm_border_apolarity_intro}
    Suppose a tensor or polynomial $F$ has border rank at most $r$.
    Then there exists a (multi)homogeneous ideal $I \subset S[X]$ such that:
    \begin{itemize}
     \item $I\subset \apolar$,
     \item for each multidegree $D$ the $D$-th graded piece $I_D$ of $I$
           has codimension (in $S[X]_D$) equal to $\min(r, \dim S[X]_D)$.
    \end{itemize}
    In addition, if $G$ is a group acting on $X$ and preserving $F$, 
      then there exists an $I$ as above which in addition is invariant under 
      a Borel subgroup of $G$.
\end{thm}

A stronger version of the first part of the theorem is presented in 
  Theorem~\ref{thm_nonsaturated_apolarity}. 
It involves more conditions on the ideal $I$, 
  and  the claim is an ``if and only if'' statement.
The second part is explained in the  Fixed Ideal Theorem
(Theorem~\ref{thm_G_invariant_bVSP}).

More explicitly, in 
  Sections~\ref{sec_multihomogeneous_coordinates} and \ref{sec_apolarity} 
  we construct a projective algebraic variety, 
  called a \emph{slip (Scheme of Limits of Ideals of Points)},
  parameterising all multigraded ideals in $S[X]$ 
  that are relevant to the construction of secant varieties.
Moreover the slip contains a dense subset of ideals representing
  $r$ distinct points of $X$.
  
The definition of the secant variety involves the closure.
It would be nice to ``get rid'' of this closure in order 
  to have a uniform description of points in the secant variety.
Such a uniform description is classical for $r=2$
(the first non-trivial secant variety), where for each point $p \in \sigma_2(X)$, either $p$ is of rank at most $2$ (that is, $p$ is on $X$ or on an honest secant line joining two distinct points of $X$), or $p$ is on a tangent line to $X$.
For tensors of border rank $3$ the situation starts
   to be  more complicated~\cite{landsberg_jabu_third_secant}.
Describing the fourth secant variety in general seemed hopeless so far, except in the case of the Veronese variety
   \cite{landsberg_teitler_ranks_and_border_ranks_of_symm_tensors}, 
   \cite{bernardi_gimigliano_ida}, \cite{ballico_bernardi_4th_secant_of_Veronese}.
Initial attempts involved spans of \emph{finite smoothable schemes}; 
  see \cite{jabu_jeliesiejew_finite_schemes_and_secants} 
  for an overview.
That is, 
\[
  \sigma_r(X) = \overline{\bigcup \set{ \linspan{R} 
     \mid R\subset  X, R \text{ is a finite smoothable scheme of length }\leqslant r}}.
\]
Roughly, a set $\setfromto{p_1}{p_r}$ of $r$ distinct points  
  is a smooth scheme of length $r$, 
  and a smoothable scheme is a limit
  (in the sense of algebraic geometry) 
  of such a collection of $r$ points.
This approach helps significantly 
  to get rid of the closure in several cases,
  but it does not work in general, as discussed 
  in \cite{bernardi_brachat_mourrain_comparison}
  and \cite{nisiabu_jabu_smoothable_rank_example}.
We briefly review two relevant examples 
  in Subsections~\ref{sec_tensor_of_border_rank_3} 
  and \ref{sec_wild_cubic}.
  
The method we propose works in all cases.
Our naturally constructed slip 
  (scheme of limits of ideals of points) 
  is a parameter space 
  for all possible limits that appear when considering 
  the closure in the definition of the secant variety.
Also, given a tensor or polynomial $F$ 
  one can define the  projective variety of all solutions 
  to the border rank problem, 
  by analogy to \emph{$\mathit{VSP}$, or Varieties of Sums of Powers} 
  (see Section~\ref{sec_VSP} for a discussion and references).
Then, again, with our approach the closure is not needed to 
  define the border version of $\mathit{VSP}$, denoted $\bVSP$.
This makes it possible to heavily exploit the group actions 
  on $\bVSP$ and in cases of tensors with large groups of symmetries, one can sometimes reduce the problem of determining the border rank to a problem of checking a finite collection
  of ideals, or several families of small dimension.

In Sections~\ref{sec_applications}--\ref{sec_monomials}
   we review applications of our method.
After an initial discussion of three previously known results 
  from the perspective of border apolarity, 
  we present two new applications.

The first application concerns tensors of minimal border rank.
It is a necessary criterion for such tensors that seems 
  to be different from existing criteria.
\begin{thm}\label{thm_minimal_border_rank_intro}
   Suppose $F\in \CC^n\otimes \CC^n\otimes \CC^n$ is a concise tensor of border rank $n$. 
   Then the multigraded ideal $\apolar \subset S[X]$
     has at least 
     $n-1$ linearly independent minimal generators 
     in degree $(1,1,1)$.
\end{thm}
This statement in a more general setting (for partially symmetric tensors) is  
  Theorem~\ref{thm_minimal_border_rank_product_of_Pa_s}, 
  and a related result in a similar direction is  
  Theorem~\ref{thm_minimal_border_rank_general_case}.

As the second application 
we calculate the border rank of monomials
 in $S^{d_1}\CC^3 \otimes S^{d_2}\CC^2 \otimes S^{d_3}\CC^2\otimes \dotsb \otimes S^{d_k}\CC^2$.
\begin{thm}\label{thm_monomials_intro}
   Suppose 
   \[
      F= x_0^{a_0}x_1^{a_1}x_2^{a_2}\otimes 
        y_0^{b_0}y_1^{b_1}\otimes \dotsm
        \otimes z_0^{c_0}z_1^{c_1} \in 
      S^{d_1}\CC^3 \otimes S^{d_2}\CC^2 \otimes 
        \dotsb \otimes S^{d_k}\CC^2
   \]
   for $a_0\geqslant a_1\geqslant a_2$ and $b_0\geqslant b_1$, \dots,  
        $c_0\geqslant c_1$.
   Then
   \[
     \borderrank[\PP^2\times \PP^1 \times \dotsb \times \PP^1]{F}= 
      (a_1 +1)(a_2+1)(b_1+1)\dotsm(c_1+1).
   \]
\end{thm}
This theorem is shown as Example~\ref{ex_monomials_P2xP1xP1___}, 
   which is a consequence of a more general 
   Theorem~\ref{thm_monomials}.
We also calculate or provide new lower bounds 
  for many other monomials 
  (Examples~\ref{ex_mono_222}--\ref{ex_almost_unbalanced}, \ref{ex_br_mono_4443_lower_bound}), 
  focusing on the Veronese case.

In addition to the applications presented in this article, 
  Theorem~\ref{thm_border_apolarity_intro} 
  is applied in \cite{conner_harper_landsberg_border_apolarity_I}
  to provide lower bounds for (and in some cases, calculate the exact value of)
  the border ranks of matrix multiplication tensors, 
  the polynomial which is the determinant of
  a generic square matrix, 
  and other tensors with large symmetry groups.
   
\subsection{Overview}

Throughout  the article we work in 
  the  general setting of a (smooth projective) 
  toric variety $X$ embedded equivariantly 
  into a projective space via
  a complete linear system.
In particular, this approach includes all Segre-Veronese varieties.
In Section~\ref{sec_multihomogeneous_coordinates}
  we review the main language of this article including 
  the Cox ring $S[X]$, multigraded ideals and corresponding subschemes
  (or subvarieties), families of ideals, 
  and multigraded Hilbert schemes.

In Section~\ref{sec_apolarity} we first recall
  multigraded apolarity and explain in detail 
  the objects appearing in \eqref{equ_duality_of_apolarity}.
Then for each $r$ we distinguish a single irreducible component
  of the multigraded Hilbert scheme and call it a slip (Scheme of Limits of Ideals of Points).
We show its relation to the secant varieties and prove the central result of this article, that is, the border apolarity 
(Theorem~\ref{thm_nonsaturated_apolarity}).
In Section~\ref{sec_VSP} we turn our attention to the set
  of solutions to the border rank problem and we show that 
  it forms a nice projective variety $\bVSP$, allowing one to exploit invariant theory to simplify the search for such solutions.
In Sections~\ref{sec_applications}--\ref{sec_monomials}
  we discuss examples
  and applications.

Several statements in this article can be strengthened 
  and generalised at the cost of becoming more technical.
Section~\ref{sec_further_research} adumbrates these claims,
  while the details will be 
  explained in a separate paper in preparation.
Moreover, introducing the theory of border apolarity 
  opens a path to a series of new problems
  as summarised in Section~\ref{sec_further_research}.

\subsection*{Acknowledgements}

We are  grateful to Joachim Jelisiejew and 
  Mateusz Micha{\l}ek for  numerous technical hints,
  to Joseph Landsberg and Zach Teitler for discussions 
  and numerous comments, and to Maciej Ga{\l}{\k a}zka, 
  Giorgio Ottaviani and Kristian Ranestad for remarks on 
  the initial version of this article.
We thank Jerzy Trzeciak for his comments on English grammar and style.
In addition, the article was developed in parallel to \cite{conner_harper_landsberg_border_apolarity_I}.
In our work we principally build the theoretical framework,
while \cite{conner_harper_landsberg_border_apolarity_I} 
provides important and relevant applications 
and a down-to-earth algorithm using the theory.
We are grateful to the authors 
  of \cite{conner_harper_landsberg_border_apolarity_I}
  for sharing the subsequent versions of their results 
  and the initial drafts of the article.
The computer algebra program Magma \cite{magma} was helpful in the calculation of explicit examples.

W.~Buczy{\'n}ska is supported by the Polish National Science Centre project ``Algebraic Geometry: Varieties and Structures'', 2013/08/A/ST1/00804.
W.~Buczy{\'n}ska and J.~Buczy{\'n}ski are also supported by the Polish National Science Centre project ``Complex contact manifolds and geometry of secants'', 2017/26/E/ST1/00231.
The article is also a part of the activities of the AGATES research group.

\section{Multihomogeneous coordinates}\label{sec_multihomogeneous_coordinates}

In this article we  assume 
  for the sake of  clarity that the base field $\kk$ 
  is the field of complex numbers $\CC$.
In Subsection~\ref{sec_other_base_fields} we briefly discuss
   generalisations to other base fields.

\subsection{Toric varieties and multihomogeneous ideals}
   \label{sec_setting_toric_vars_and_mutligr_Hilb}
Let $X$ be a smooth toric projective variety of dimension $n$ with Picard group $\Pic(X) \simeq \ZZ^w$.
We denote by $S=S[X] =  \bigoplus_{D \in \Pic(X)}H^0(\ccO_X(D))$ its Cox ring,
which is naturally a  $\Pic(X)$-graded polynomial ring  $\kk[\fromto{\alpha_1}{\alpha_{n+w}}]$, see~\cite[Prop.~5.3.7a]{cox_book}.
This grading is \emph{positive} in the sense that there is  only one monomial of degree zero, namely $1$.
We fix a very ample line bundle $L$ and from now on we consider  $X \subset \PP(H^0(L)^*)$ as an embedded projective variety.
We denote by $\Irrel{X}=\sqrt{(H^0(L))}$  the \emph{irrelevant ideal} of $S[X]$.

\begin{example}\label{ex_Veronese}
  Consider $X\simeq \PP^n$ and $L=\ccO_{\PP^n}(d)$.
  Thus $\Pic(X)=\ZZ$ and $X$ is embedded via the degree $d$ Veronese map in
  $\PP^N= \PP(H^0(\ccO_{\PP^n}(d))^*) = \PPof{S^d \CC^{n+1}}$.
  The Cox ring of $X$ is the $\ZZ$-graded polynomial ring
  $    S[\PP^n]=\CC[\alpha_0,\ldots,\alpha_n] $
  with $\deg(\alpha_i)=1$.
  Here $\Irrel{X} =(\alpha_0,\ldots,\alpha_n)$ is the unique homogeneous maximal ideal.
\end{example}

\begin{example}\label{ex_Segre}
  Consider $X\simeq \PP^a\times \PP^b \times \PP^c$ and
  $L=\ccO_X(1,1,1):=\ccO_{\PP^a}(1)\boxtimes \ccO_{\PP^b}(1)\boxtimes \ccO_{\PP^c}(1)$.
  Thus $\Pic(X)=\ZZ^3$ and $X$ is embedded via the Segre map in
  $\PP^N= \PP(H^0(\ccO_X(1,1,1))^*) =
              \PPof{ \CC^{a+1}\otimes \CC^{b+1}\otimes\CC^{c+1}}$.
  The Cox ring of $X$ is the $\ZZ^3$-graded polynomial
  ring
  \[
    S[\PP^a\times \PP^b \times \PP^c]=\CC[\alpha_0,\ldots,\alpha_a,\beta_0,\dotsc,\beta_b, \gamma_0,\dotsc,\gamma_c ]
    \]
  with $\deg(\alpha_i)=(1,0,0)$, $\deg(\beta_i)=(0,1,0)$, $\deg(\gamma_i)=(0,0,1)$.
  Here 
  \begin{align*}
     \Irrel{X} &=(\alpha_0,\ldots,\alpha_a) 
     \cap  (\beta_0,\dotsc,\beta_b)
     \cap (\gamma_0,\dotsc,\gamma_c)\\ 
     &= (\alpha_i \beta_j \gamma_k \mid 
     i \in \setfromto{0}{a}, j \in \setfromto{0}{b},k \in \setfromto{0}{c}).
  \end{align*}
\end{example}

\begin{notation}\label{not_graded_piece}
   Whenever we  consider a ring $S$ graded by $\Pic (X)$, 
      a homogeneous ideal $I \subset S$, 
      or a graded $S$-module $M$,
      and $D\in \Pic (X)$,
      we denote by $S_D$, $I_D$, $M_D$  the $D$-th graded piece of $S$, $I$, or $M$, respectively.
      Analogously, for sheaves of graded rings, homogeneous ideals or graded modules,
       $\cdot_D$ will also mean the $D$-th graded piece. 
\end{notation} 
Note that $\Irrel{X}$ is a homogeneous ideal and it does not depend on the choice of $L$.
A classical quotient interpretation of the Cox ring is that
\[
 X=\left.\bigl(\AAA^{n+w}\setminus \Spec(S/\Irrel{X})\bigr)\right/(\CC^*)^w
\]
(see \cite[Thm~5.1.11]{cox_book}).
Here $\AAA^{n+w}=\Spec S$,  
$\Spec(S/\Irrel{X})$ is the zero locus of $\Irrel{X}$ in $\AAA^{n+w}$,
and $(\CC^*)^w= \Hom(\Pic(X),\CC^*)$ 
is the torus acting diagonally on 
  $\AAA^{n+w}=\Spec S$ with weights corresponding to the degrees of the variables.
This quotient construction gives rise to the toric ideal-subscheme correspondence, which we briefly describe now.

Any homogeneous ideal $I\subset S$ defines its zero scheme
$\zeroscheme(I) \subset X$ \cite[paragraph before Thm~3.7]{cox_homogeneous}.
In particular,  $\zeroscheme(\Irrel{X}) = \emptyset$.
We say that a homogeneous ideal $I$ is \emph{saturated} if  $I=(I:\Irrel{X})$ or equivalently 
$I_D= \set{s\in H^0(D) \mid s_{|_{\zeroscheme(I)}}=0}$
for any line bundle $D\in\Pic(X)$
where $I_D$ is as in Notation~\ref{not_graded_piece}.
There is a one-to-one correspondence between saturated homogeneous ideals in $S$ and subschemes of $X$.
In particular, the saturated ideal corresponding to the subscheme $R \subset X$ is denoted by $\idealof(R)$.
Let $I^{\sat}$ be the saturation of $I$,
  that is, the smallest saturated ideal containing $I$,
  also obtained by successively replacing $I$ with $(I:\Irrel{X})$, until it stabilises.  
See~\cite[\S2.1]{brown_jabu_maps_of_toric_varieties} for a more general situation and more details. Note that since here we assume $X$ is smooth and projective, 
the divisor class group used in
\cite{brown_jabu_maps_of_toric_varieties}
 is torsion free and equal to $\Pic(X)$.

Similarly, in the relative setting, if $B$ is another variety (or scheme) over $\CC$,
  then we have a correspondence (not bijective) between 
  subschemes $\ccZ \subset X\times B$ and homogeneous ideal sheaves $\ccI\subset S\otimes \ccO_B$: 
  to a subscheme $\ccZ$ in grading $D\in \Pic(X)$ 
  we assign the sheaf 
  $\idealof(\ccZ)_D= \set{s\in H^0(D)\otimes \ccO_B \mid s_{|_{\ccZ}}=0}$, and then $\idealof(\ccZ)= \bigoplus_{D\in \Pic(X)} \idealof(\ccZ)_D \subset S\otimes \ccO_B$.
  In the other direction, 
     to a sheaf of ideals $\ccI$ we assign 
     the scheme 
  \[
    \zeroscheme(\ccI) = \left.\Bigl(\SheafySpec (S\otimes \ccO_B/ \ccI) 
      \setminus \bigl(\Spec (S/ \Irrel{X})\times B\bigr)\Bigr) \right/ \Hom(\Pic(X),\CC^*).
  \]
Again, as above, we say that the family of ideals $\ccI$ is \emph{saturated} if  
$\ccI = \bigl(\ccI: (\Irrel{X}\otimes \ccO_B)\bigr)$.   
For a closed point $b\in B$ 
   we set $\ccI_b \subset S$ to be the fibre ideal,
   that is, $\ccI_b =\ccI \otimes_{\ccO_B} \ccO_b$.
\begin{prop}\label{prop_semicontinuities_for_families_of_ideals}
   Let $\ccZ\subset X \times B$ be a subscheme 
   and $\ccI\subset S\otimes \ccO_B$ be
   a homogeneous ideal sheaf, as above. Pick any $D\in \Pic(X)$. Then
   \begin{enumerate}
    \item\label{item_ideal_of_family_is_saturated}
          $\idealof(\ccZ)$ is saturated.
    \item\label{item_Hilb_func_of_family_of_ideals_is_upper_semicont}
          If $B$ is a variety, 
             then $\dim \,(S/\ccI_b)_D$ is an upper semicontinuous function of $b$, 
             that is, for each integer $r$ the set 
             $\set{b\in B\mid \dim\,(S/\ccI_b)_D \geqslant r}$ is a closed subset of $B$.
    \item\label{item_Hilb_func_of_saturations_is_lower_semicont}
     If $B$ is a variety and 
             $\ccZ$ is flat over $B$,
             then $\dim \left(S/\left(\idealof(\ccZ)_b\right)^{\sat}\right)_D$, which is equal to $\dim \bigl(S/\idealof(\ccZ_b)\bigr)_D$, 
             is a lower semicontinuous function of $b$, that is, for each integer $r$ the set 
             $\set{b\in B \mid \dim \left(S/\left(\idealof(\ccZ)_b\right)^{\sat}\right)_D \leqslant r}$ is a closed subset of $B$.
   \end{enumerate}
\end{prop}
\begin{proof}
   The first item is straightforward. 
   The second item is \cite[Exercise~II.5.8(a)]{hartshorne}.
   The third item follows from part (b) of the same exercise, which shows that 
   the dimension of 
   $\left(S/\left(\idealof(\ccZ)_b\right)\right)_D$
   is constant (independent of $b$).
   Then use $\idealof(\ccZ)_b \subset \left(\idealof(\ccZ)_b\right)^{\sat}$ to conclude.
\end{proof}

\begin{example}
   A classical case of the dichotomy of semicontinuity 
   as in Proposition~\ref{prop_semicontinuities_for_families_of_ideals}\ref{item_Hilb_func_of_family_of_ideals_is_upper_semicont} 
   and \ref{item_Hilb_func_of_saturations_is_lower_semicont}
   is the case of four points moving on a projective plane.
   So let $B=\AAA^2 = \Spec\CC[s,t]$, $X = \PP^2$, 
     $S=S[X]=\CC[\alpha_0,\alpha_1,\alpha_2]$, $D=\ccO_{\PP^2}(2)$.
   Consider four distinct points of $X$ parameterised by~$B$:
   \[
           \chi_1 = [1,0,0], \quad \chi_2 = [0,1,0], 
     \quad \chi_3 = [1,1,s], \quad \chi_4 = [1,-1,t].
   \]
   Here $\chi_1$ and $\chi_2$ are independent of $s$ and $t$. 
    For any $s,t$ the four points are distinct, and for $s=t=0$, they are collinear, while for all other parameters, they are linearly nondegenerate.
   Let $\ccZ= \set{\chi_1, \chi_2, \chi_3, \chi_4}$. 
   Then $\ccZ\to B$ is flat, and 
   \begin{align*}
      \dim\Bigl(S/(\idealof\bigl(\ccZ)_{s,t}\bigr)\Bigr)_2 &=
                                \begin{cases}
                                   4 & \text{for } (s,t) \ne (0,0),\\
                                   5 & \text{for } (s,t) = (0,0),
                                \end{cases} \text{ and}\\
      \dim\Bigl(S/\bigl(\idealof(\ccZ_{s,t})\bigr)\Bigr)_2 = 
      \dim \Bigl(S/\bigl(\idealof(\ccZ)_{s,t}\bigr)^{\sat}\Bigr)_2 &=
                               \begin{cases}
                                   4 & \text{for } (s,t) \ne (0,0),\\
                                   3 & \text{for } (s,t) = (0,0).
                                \end{cases}
   \end{align*}
   That is, the two numbers agree for generic $(s,t)$, but at the special point $(s,t)=(0,0)$ the top number goes up, while the bottom one goes down.
   As the difference in the layout of parentheses
     might be hard to spot, 
   we stress that to calculate the top number
   we specialise the ideal of the whole family $\ccZ$ to $(s,t)$,
   while in the bottom number we specialise $\ccZ$ to $(s,t)$, and only then do we calculate the ideal.
\end{example}

We also mention a flatness condition 
for families of homogeneous ideals, which is easy to apply: 
essentially the flatness is equivalent
to the Hilbert function being independent of the ideal in the family.
This is analogous to \cite[Thm~III.9.9]{hartshorne}, 
  where for families of projective schemes 
  the flatness is equivalent to the Hilbert polynomial being independent of the ideal.

\begin{lemma}\label{lem_flatness_is_constant_Hilert_function}
    Let $B$ be a variety, and $\ccI\subset S\otimes \ccO_B$ a family of homogeneous ideals.
    Then the following conditions are equivalent:
    \begin{itemize}
     \item $\ccI$ is flat over $B$,
     \item $S\otimes \ccO_B/ \ccI$ is flat over $B$,
     \item for each degree $D\in \Pic(X)$,
     $\dim (\ccI_b)_D$ does not depend on the choice of the point $b\in B$,
     \item for each degree $D\in \Pic(X)$, 
           $\dim ((S\otimes \ccO_B/ \ccI)_b)_D$ 
           does not depend on the choice of the point 
           $b\in B$.
    \end{itemize} 
\end{lemma}

\subsection{Multigraded Hilbert scheme}

Consider a function $h\colon \Pic(X) \to \NN$, where $\NN$ is the set of non-negative integers. 
We assume $h$ is non-zero only on effective divisors, that is, on those  $D\in\Pic(X)$  for which $S_D\ne 0$.
Let $\Hilb^h_S$ be \emph{the multigraded Hilbert scheme}, which parameterises \emph{all} the homogeneous ideals $I \subset S$
such that the Hilbert function of $S/I$ is $h$. 
We stress that in general $\Hilb^h_S$ contains points that represent both saturated and non-saturated ideals, 
hence it is not necessarily equal to any (standard) Hilbert scheme, even in the standard  case, when $X=\PP^n$ is a projective space. 
See \cite{haiman_sturmfels_multigraded_Hilb} for more on the definition and properties of the multigraded Hilbert scheme.
In particular, by \cite[Thm~1.1 and Cor.~1.2]{haiman_sturmfels_multigraded_Hilb} the scheme $\Hilb^h_S$ is projective, since the grading is positive in our setting.    

\begin{rmk}
   Note that the name \emph{multigraded} Hilbert scheme
   proposed by Haiman and Sturmfels
   might be a little confusing, because you could expect that
   if you specialise the multigraded case to the single graded case,
   you obtain the standard Hilbert scheme, while this is not the case.
   The main difference coming from the adjective ``multigraded'' is that the multigraded Hilbert scheme
   parameterises \emph{ideals} with a fixed Hilbert \emph{function}, as opposed to the standard Hilbert scheme,
   which parameterises \emph{subschemes} 
   with a fixed Hilbert \emph{polynomial}.
\end{rmk}

Note that depending on the grading of $S$ 
  and on the Hilbert function $h$, 
  the multigraded Hilbert scheme $\Hilb_S^h$ might be (non-)empty, (ir)reducible, (dis)connected, \mbox{(non-)}re\-du\-ced. 
In this article we only consider the reduced structure of 
 $\Hilb_S^h$,
  that is, we think of $\Hilb^h_S$ 
  as a finite union of projective varieties 
  (sometimes also called a reducible variety).
We  denote this (possibly reducible) variety 
   by $\reduced{(\Hilb^h_S)}$.
Thus each closed point of $\Hilb^h_S$ 
   or $\reduced{(\Hilb^h_S)}$
   represents a homogeneous ideal $I\subset S$
   and in such a situation we simply write $I\in \Hilb^h_S$.
   
We consider the subset 
$\Hilb_S^{h, \sat} \subset \Hilb^h_S$, 
consisting of the closed points 
representing saturated ideals.
This set may be empty, dense, or neither. 
It can be shown that this is (the set of closed points of) a Zariski open subscheme 
 (see Subsection~\ref{sec_saturation_open} for a brief discussion).
 Here we prove a weaker statement which is sufficient for the results of the article:
 in each irreducible component of $\Hilb^h_S$ the subset of saturated ideals is either empty or dense. 
For this purpose we need the following definition.

\begin{defin}
  For an irreducible variety $Y$, we say that a property $\ccP$ is satisfied for a \emph{very general} point of $Y$ if it is satisfied 
  for every point outside of a countable union of proper Zariski closed subsets of $Y$.
\end{defin}

Since we work over $\CC$, 
by the Baire category theorem, 
if $\ccP$ holds for a very general point of $Y$, 
then the set of points in $Y$ that satisfy $\ccP$ is dense in $Y$ 
in the analytic topology of $Y$, and therefore also in the Zariski topology of $Y$.

\begin{prop}\label{prop_saturated_is_open}
Suppose $\sH \subset \reduced{(\Hilb^h_S)}$ is an irreducible component.
Then either $\sH\cap \Hilb_S^{h, \sat}$ is empty or it contains a very general point of $\sH$
(in particular, in the latter case, the intersection is dense in $\sH$). 
\end{prop}
\begin{prf}
   Suppose  $\sH\cap \Hilb_S^{h, \sat}$ 
      is non-empty and take a saturated ideal $J\in \sH$.
   Pick any $D\in \Pic(X)$ and denote
   \[
      \sH_{D} := \set{I \in \sH \mid \dim (S/I^{\sat})_D < h(D)}.
   \]
   This is a Zariski closed subset of $\sH$ by
   Proposition~\ref{prop_semicontinuities_for_families_of_ideals}\ref{item_Hilb_func_of_saturations_is_lower_semicont}. 
   Since $\dim (S/J^{\sat})_D = \dim (S/J)_D = h(D)$, thus  $\sH_D \ne \sH$.
   We have
   \[
     H\cap \Hilb_S^{h, \sat} = \sH \setminus \bigcup_{D\in \Pic(X)} \sH_{D}
   \]
   and thus $H \cap \Hilb_S^{h, \sat}$ is the complement of a union of countably many  Zariski closed strict subsets.
\end{prf}
Note that it is not enough to remove only $H_D$ 
   for $D$ in some minimal generating set of $\Pic(X)$ 
   or of the effective cone.
Already in the case of the Veronese threefold $X= \PP^3$ as in 
   Example~\ref{ex_Veronese},
   if $h = (1,4,6,6,6,\dotsc)$,
   then for a component $\sH \subset \Hilb_S^{h}$ we have 
   $\sH\subset \Hilb_S^{h, \sat} = \sH \setminus (\sH_{1} \cup \sH_{2})$,
   with $\sH_{1}$ being the set of ideals  whose saturation has a linear form,
   while $\sH_{2}$ is the set of ideals whose saturation has at least five independent quadrics.
At the other extreme, if $h=(1,2,2,2,\dotsc)$,
   then all the homogeneous ideals in $S[\PP^n]$ with Hilbert function $h$ 
   are saturated, so we do not need to take out any $\sH_D$, in particular, the set of such $D$'s does not generate $\Pic(X)$.

\section{Apolarity theory on toric varieties}\label{sec_apolarity}

\subsection{Rank, border rank and multigraded apolarity}
\label{sec_apolarity_for_rank}
Following Ga{\l}{\k a}zka~\cite{galazka_mgr}, we recall the setting for multigraded apolarity on $X$.
Recall that the Cox ring $S$ is $\kk[\fromto{\alpha_1}{\alpha_{n+w}}]$, where $\alpha_i$ are homogeneous generators of $S$ 
   which correspond to primitive torus invariant divisors of $X$.
We let $\widetilde{S}:=\kk[\fromto{x_1}{x_{n+w}}]$ be the dual graded polynomial ring, which we consider as a divided power algebra (with $x_i^{(d)} = \frac{1}{d!}x_i^{d}$).
It is also a graded $S$-module with the following action:
\begin{equation}\label{equ_apolarity_hook_definition}
   \alpha_i \hook \left(x_1^{(a_1)}\cdot x_2^{(a_2)} \dotsm x_{n+w}^{(a_{n+w})}\right) = 
   \begin{cases}
           x_1^{(a_1)}\dotsm x_i^{(a_i-1)} \cdots x_{n+w}^{(a_{n+w})} & \text{if } a_i>0,\\
           0 & \text{otherwise.} 
   \end{cases}
\end{equation}
The grading in $\widetilde{S}$ is given by writing 
\[
  \widetilde{S}= \bigoplus_{D\in \Pic(X)} H^0(D)^*
\]
where the duality is given by~\eqref{equ_apolarity_hook_definition}.
Thus the coordinate free expression of the apolarity action $\hook$ is the following.
  Let 
  $F\in H^0(D_1)^* = \widetilde{S}_{D_1}$ 
  and $\Theta\in H^0(D_2)= S_{D_2}$
  for some $D_1 ,D_2 \in \Pic(X)$. 
Then $\Theta\hook F \in \widetilde{S}_{D_1-D_2}$
  is defined as the functional
  $H^0(D_1-D_2)\to \CC$  given by
\[
   (\Theta\hook F) (\Psi) = F(\Theta\cdot \Psi),
\]
 where $\Psi \in H^0(D_1-D_2)$ is arbitrary, $\Theta\cdot \Psi \in S_{D_1}$
   is the product in the ring $S$, 
   and $F(\ldots)$ is the evaluation 
   of the functional $F$. 
In particular, we have the following natural property of $\hook$:
   
\begin{prop}[Apolarity fixes $X$]
  Suppose $D_1$ and $D_2$ are two effective divisors.
  Denote by $\varphi_{|D_i|}\colon X \dashrightarrow \PPof{H^0(D_i)^*}$ the rational map determined by the complete linear system of $D_i$.
    Let $\hat{X}_i\subset H^0(D_i)^*$ 
    be the affine cone of the closure of the image of $X$ under $\varphi_{|D_i|}$.
  The apolarity action
        \[
          \hook \colon S_{D_1-D_2} \otimes \widetilde{S}_{D_1} \to \widetilde{S}_{D_2}
        \]
      preserves $X$,
      that is, for all $\Theta\in S_{D_1-D_2}$ and $p\in \hat{X}_1$
      we have $\Theta \hook p\in\hat{X}_2$.
      Moreover, if $\chi\in X$ is such that 
         $[p]=\varphi_{|D_1|}(\chi)$, 
         $\chi$ is outside of the base locus 
         of both divisors $D_1$ and $D_2$,
         and $\Theta \hook ( \varphi_{|D_1|}(\chi)) \neq 0$,
         then
         $\varphi_{|D_2|}(\chi)=[\Theta \hook p]$.
\end{prop}
   Perhaps it is easier to understand
   the above statement when both $D_1$ and $D_2$ are very ample, so that both $\varphi_{|D_i|}$ 
   are embeddings of $X$ into different projective spaces
   and $\Theta$ is non-zero.
   Then the rational (linear) map  
    $\PPof{H^0(D_1)^*} \dashrightarrow \PPof{H^0(D_2)^*}$ of projective spaces
    determined by $\Theta\hook\cdot$ restricts to the identity map on 
    $X \dashrightarrow X$ (wherever defined).
    Note that this expression of the restriction exploits both embeddings
    $\varphi_{|D_i|}$.
\begin{prf}
   If $\Theta=0$, there is nothing to prove, 
      thus we assume $\Theta\ne 0$.
      Since the map $\Theta \hook \cdot$ is linear, it is continuous in the Zariski topology,
      and so is its restriction to $\hat{X}_1$.
   Thus it is enough to prove the claim for general 
     $p\in \hat{X}_1$. 
   More precisely, we will assume that
   \begin{enumerate}
    \item \label{item_proof_hook_preserves_X_point_in_the_image}
       there exists a point $\chi \in X$ such that $[p]\in \PPof{H^0(D_1)^*}$ is the image of $\chi$, 
    \item \label{item_proof_hook_preserves_X_Theta_does_not_vanish}
        the section $\Theta\in H^0(\ccO_X(D_1-D_2))$ does not vanish at $\chi$, and
    \item  \label{item_proof_hook_preserves_X_not_all_sections_vanish}
        $\chi$ is not in the base locus of $D_2$.
   \end{enumerate}
   Each of 
   \ref{item_proof_hook_preserves_X_point_in_the_image}--\ref{item_proof_hook_preserves_X_not_all_sections_vanish}
   is a non-empty and Zariski open condition on $\chi\in X$, 
   and thus the image of the intersection of these conditions
      is dense in $\hat{X}_1$.
   Equivalently to \ref{item_proof_hook_preserves_X_point_in_the_image}, 
      the hyperplane $p^{\perp} \subset H^0(D_1)$
      consists of sections vanishing at $\chi$.
   Thus 
   $\set{\Psi\in H^0(\ccO_X(D_2)) \mid \Theta\Psi \in p^{\perp}}$ 
    is a linear preimage of a hyperplane, hence either a hyperplane or the whole $H^0(\ccO_X(D_2))$.
   That is, it is equal to $(p')^{\perp}$
     for some $p' \in \tilde{S}_{D_2}$ (well defined up to rescaling).
   Moreover,
     \ref{item_proof_hook_preserves_X_Theta_does_not_vanish}
     guarantees that all sections in $(p')^{\perp}$ 
     vanish at $\chi$.
     If $p'=0$, then all sections of $D_2$ vanish at $\chi$,
     a contradiction with \ref{item_proof_hook_preserves_X_not_all_sections_vanish}.
   Thus $p'\ne 0$ and the image of $\chi$ in $\PPof{H^0(D_2)^*}$
     is equal to $[p']$. In particular, $p'\in \hat{X}_2$.
     It remains to observe that by the construction of $p'$,
     and the coordinate free description of $\hook$, $p' = \Theta \hook p$
     up to a non-zero rescaling of~$p'$.  
\end{prf}

\begin{defin}\label{def_all_sorts_of_things} We recall the following notions.
  \begin{enumerate}
  \item  \label{item_def_proj_linear_span} For a scheme $R$ in a projective space $\PP(V)$
    its \emph{projective linear span}, denoted $\linspan{R}$, is  the
    smallest  projective linear subspace of $\PP V$ containing~$R$.
   \item The \emph{$X$-rank} of $F$ is the minimal integer $r=\rank F$ such that $[F] \in \linspan{\fromto{p_1}{p_r}}$, 
     where $p_i$ are points in 
     $X \subset \PP(\widetilde{S}_L)$.
   \item \label{item_def_apolar_ideal}
     Let $F \in \widetilde S_L = H^0(L)^*$ and denote by $[F] \in \PP(\widetilde{S}_L)$ the corresponding point in the projective space.
     The \emph{apolar ideal} of $F$ is the homogeneous ideal $\apolar$ of $S$ annihilating $F$. 
     Explicitly, 
         \[
         \apolar=\set{\Theta\in S \mid   \Theta \hook F=0}.
         \]
   \item\label{item_def_perp}
      For a linear subspace $W \subset V$
      we denote by $W^{\perp} \subset V^*$ 
      the perpendicular space. 
  \end{enumerate}
\end{defin}
\begin{rmk}\label{rem_interactions_between_all_sorts_of_things}
  Typically, we will use~Definition~\ref{def_all_sorts_of_things}\ref{item_def_perp} for 
    a specific degree $L$ of a homogeneous ideal $I\subset S$. 
  Then  $I_L^{\perp} \subset H^0(L)^* = \widetilde{S}_L$ is the perpendicular space with respect to the duality action \eqref{equ_apolarity_hook_definition}. 
  We note the following interactions between items~\ref{item_def_proj_linear_span}, \ref{item_def_apolar_ideal}, \ref{item_def_perp}
    of Definition~\ref{def_all_sorts_of_things}: 
    \begin{itemize}
    \item For a subscheme $R\subset X \subset \PPof {H^0(L)^*}$ we have $\linspan{R} = \PPof {\idealof(R)_L^{\perp}}$. 
    \item Let $L, D \in \Pic(X)$ and $F\in S_L$. Then $\apolar_D^{\perp}=S_{L-D}\hook F \subset H^0(D)^*$. 
       In particular,  $\apolar_L^{\perp}$ is the linear span of $F$.
    \end{itemize}
\end{rmk}

The following property of apolarity is well known in the single graded setting.
\begin{prop}\label{prop_apolar_in_degree_deg_F_vs_apolar}
    Suppose $F \in \widetilde S_L$ and $I\subset S$ is a homogeneous ideal. Then
    \[
      I \subset \apolar \iff I_L \subset \apolar_L.
    \]
\end{prop}
The single graded proof~\cite[Prop.~3.4(iii)]{nisiabu_jabu_cactus} 
   works also for the multigraded case.
See also~\cite[proof of Thm~1.1]{galazka_mgr}
or \cite[Lem.~1.3]{gallet_ranestad_villamizar}.
The multigraded apolarity is the following proposition.
\begin{prop}[multigraded apolarity]\label{prop_multigraded_apolarity}
   Consider a smooth toric projective variety $X$ with Cox ring $S$ and embedded in $\PP( H^0(L)^*)=\PPof{\widetilde{S}_L}$.
   Suppose $F \in \widetilde{S}_L$ 
      and pick any subscheme $R \subset X$.
   Then $ [F] \in \linspan{R}$ if and only if $\idealof(R)\subset \apolar$. 
   In particular,  
   \begin{itemize}
       \item the $X$-rank $\rank F$ is at most $r$ if and only if there exists a radical saturated ideal $I \subset S$ 
             such that $I \subset \apolar$ and $I$ is an ideal of $r$ points. 
   \end{itemize}
\end{prop}

\begin{rmk}
The statement of multigraded apolarity  coincides with the standard 
  apolarity~\cite[Thm~5.3.B]{iarrobino_kanev_book_Gorenstein_algebras}
  in the case where $X$ is projective space in its Veronese embedding
  (Example~\ref{ex_Veronese}). 
For $X$ isomorphic to a product of projective spaces 
  (Segre-Veronese varieties)
  it appeared in 
  \cite[Thm~4.10]{teitler_lower_bound_for_generalized_ranks}.
Then it was shown in the Master Thesis of Ga{\l}{\k a}zka 
\cite[Thm~1.1]{galazka_mgr} for any $\QQ$-factorial toric projective variety, 
and later in \cite[Lem.~1.3]{gallet_ranestad_villamizar}, where  it was  again proved (and used) for smooth projective toric varieties.
It seems plausible that the analogous statement 
   can also be proven for any projective Mori Dream Space 
   (that is, a projective variety with a reasonable analogue 
    of the Cox ring, or total coordinate ring).
\end{rmk}

We now define the secant varieties and the border variant of rank. 
Our goal in Section~\ref{sec_apolarity_for_border_rank} is to generalise apolarity (Proposition~\ref{prop_multigraded_apolarity}) to the border rank.

\begin{defin}
  The \emph{$r$\textsuperscript{th} secant variety} of $X\subset \PP( H^0(L)^*)$ is the following subvariety of $\PP( H^0(L)^*)$:
\[
\sigma_r(X)=\overline{\set{[F]\in \PP( H^0(L)^*) \mid \rank F \leq r  }}.
\]
The \emph{$X$-border rank} of ${F} \in \PP( H^0(L)^*)$, denoted $\borderrank{F}$, is the smallest $r$ such that $[F] \in \sigma_r(X)$.
\end{defin}

\subsection{Apolarity for border rank}\label{sec_apolarity_for_border_rank}

For a non-negative integer $r$ define $h_{r,X}\colon \Pic(X) \to \NN$ as
\[
  h_{r,X}(D): = \min \left( r, \dim H^0(D)\right).
\]
\begin{lemma}\label{lem_hilb_function_of_very_general_tuple}
   For any tuple
   $\bar{\chi}=(\fromto{\chi_1}{\chi_r}) \in X^{\times r} = 
   \underbrace{X\times  \dotsb \times X}_{r 
   \text{ times}}$
    let $R_{\bar{\chi}}$ be the corresponding finite collection of points 
      $\setfromto{\chi_1}{\chi_r} \subset X$ 
      (ignoring  possible repetitions). 
   Then $\dim(S/\idealof(R_{\bar{\chi}}))_D \leqslant h_{r,X}(D)$ for any $D\in \Pic(X)$. 
   Moreover, for a very general tuple
   $\bar{\chi} \in X^{\times r}$ and for all $D$ 
   we have equality:
   $\dim(S/\idealof(R_{\bar{\chi}}))_D = h_{r,X}(D)$. 
\end{lemma}

It is also true that a general configuration of points 
also has the  Hilbert function $h_{r,X}$, 
not only a very general one: see Subsection~\ref{sec_saturation_open}.

\begin{prf}
   The first claim of the lemma (inequality) is clear.

   Similarly to the proof of 
   Proposition~\ref{prop_saturated_is_open},
   for each $D\in \Pic(X)$
     we can find a configuration $R\subset X$  of $r$ 
     points with $\dim\,(S/\idealof(R))_D= h_{r,X}(D)$
     and this is an open condition on $X^{\times r}$ by 
     Proposition~\ref{prop_semicontinuities_for_families_of_ideals}\ref{item_Hilb_func_of_saturations_is_lower_semicont}.
   Intersecting (countably many of) these open conditions (for all $D$) we obtain  of the second claim (equality).
\end{prf}

\begin{example}
When $X=\PP^{n}=\PPof{V}$ and $i\in \NN \subset \Pic(X)$ we get
\[
  h_{r,\PP^{n}}(i) = \min \left( r, \binom{n+i}{n}\right) = \min \left( r, \dim S^i V \right).
\]
In this case it is well known that
(in the proof of Lemma~\ref{lem_hilb_function_of_very_general_tuple})
it suffices to intersect finitely many open conditions 
(for $i=\fromto{1}{r-2}$) to obtain $h_{r,\PP^{n}}$ 
as the Hilbert function, thus the condition is not only dense,
  but also open.
\end{example}

\begin{example}
If $X=\PP^{a}\times \PP^{b}\times \PP^{c}=\PPof{A}\times \PPof{B}\times\PPof{C}$ and $(i,j,k)\in \NN^3 
\subset \Pic(X)$ 
then the generic Hilbert function of $r$ points is
\begin{align*}
  h_{r,\PP^{a}\times \PP^{b}\times \PP^{c}}(i,j,k) &= \min \left( r, \binom{a+i}{a}\binom{b+j}{b}\binom{c+k}{c} \right)
  \\
  &= \min \Bigl( r, \dim 
  \left(S^i A\otimes S^j B\otimes S^k C\right) \Bigr).
\end{align*}
\end{example}

\begin{lemma}\label{lem_set_theoretical_union}
   Suppose $Y$ is a set and $Y= Y_1  \cup \dotsb \cup Y_k$
     for some subsets $Y_i \subset Y$.
   Suppose $Z\subset Y$ is a subset such that for any  
     $z, z'\in Z$ there exists  $i\in \setfromto{1}{k}$ 
       such that both $z$ and $z'$ are in $Y_i$.
  Then $Z\subset Y_j$ for some $j\in \setfromto{1}{k}$.
\end{lemma}
\begin{proof}
  If $k=1$, then there is nothing to prove.
  For each $z\in Z$ let $\operatorname{comp}(z)\in 2^{\setfromto{1}{k}}$
    be  $\set{i \mid z\in Y_i}$.
  Define $l := \min\set{ \# \operatorname{comp}(z) : z \in Z}$.
  If $l=k$, then the claim is proved.

  Otherwise, pick $z_0$ such that $\# \operatorname{comp}(z_0) = l$. 
  Let $Y' := \bigcup \set{Y_i \mid i \in \operatorname{comp}(z_0)}$.
  By the assumptions of the lemma, $Z\subset Y'$.
  Thus, we can replace $Y$ with $Y'$ which is a union of fewer subsets 
   $Y_i$, and argue by induction on $k$.
\end{proof}

Let $\Sip_{r,X} \subset \reduced{(\Hilb^{h_{r,X}}_S)}$ 
  be the subset consisting of saturated ideals of 
  $r$ distinct points in $X$.

\begin{prop}\label{prop_motivation_for_h_r_X}
There is a unique component of the multigraded Hilbert scheme 
   $\reduced{(\Hilb^{h_{r,X}}_S)}$, 
   that contains  $\Sip_{r,X}$ as a dense subset.
\end{prop}
\begin{proof}
   We must show that
   \begin{itemize}
     \item $\Sip_{r,X}$ is contained in 
            a single irreducible component 
            of the multigraded Hilbert scheme, and
     \item $\Sip_{r,X}$ is dense in that component.
   \end{itemize}
   To prove the first item, pick two points
      $I,I'\in \Sip_{r,X}$, that is, two saturated ideals of $r$-tuples of points, 
      $I= \idealof(\setfromto{\chi_1}{\chi_r})$,
      $I'= \idealof(\setfromto{\chi'_1}{\chi'_r})$, having the Hilbert function $h_{r,X}$.
   Pick a smooth integral curve $B$ 
      which can be used to connect $\chi_i$ 
      to $\chi'_i$ for any $i$.
   That is, pick morphisms $\psi_i\colon B \to X$ 
      and two points $b, b'\in B$  
      such that $\psi_i(b)=\chi_i$ 
      and $\psi_i(b')=\chi'_i$.
   (Since $X$ is a toric variety, 
      it is rationally connected, 
      so it is enough to take $B=\AAA^1=\Spec \CC[t]$.)
   
   Consider the sheaves of homogeneous ideals 
      $\ccJ_i\subset S\otimes \ccO_B$ defining 
      $(\psi_i \times \id_B)(B) \subset X\times B$.
   Each $S\otimes \ccO_B /\ccJ_i$ is flat 
     and has Hilbert function $h_{1,X}$.
   The affine zero-set of $\ccJ_i$ in $\Spec S \times B$
     is reduced and irreducible.
   Let 
   \[
      \ccJ = \textstyle\bigcap_{i=1}^r \ccJ_i 
      = \idealof\Bigl(\textstyle\bigcup_{i=1}^{r} 
        (\psi_i \times \id_B)(B)\Bigr).
   \]
   In particular, $\Spec (S\otimes \ccO_B /\ccJ)$
     is flat over $B$ by \cite[Prop.~III.9.7]{hartshorne}.
   Therefore, the algebra of each fibre of 
     $\Spec (S\otimes A / \ccJ) \to B$
     has a constant Hilbert function $h$ (see Lemma~\ref{lem_flatness_is_constant_Hilert_function}).
   By construction ($\ccJ$ is an intersection of $r$ ideals, 
     each has codimension at most $1$ in each degree)
     we must have $h \leqslant h_{r,X}$ (for each argument $D\in \Pic(X)$)
     and on the other hand 
     there are fibres (over $b$ and $b'$) that have 
     Hilbert function at least $h_{r,X}$ 
     (their saturations have Hilbert function $h_{r,X}$). 
   Thus $h=h_{r,X}$, 
      and the flat family of ideals determines 
      a morphism $B \to \reduced{(\Hilb^{h_{r,X}}_S)}$ 
      connecting $I$ and $I'$, 
      and showing that they are in 
      the same irreducible component.
  Since the choice of $I$ and $I'$ was arbitrary, 
      it follows that all of $\Sip_{r,X}$ is contained in a single irreducible component $\ccH$ by Lemma~\ref{lem_set_theoretical_union}.
      
   The proof of the second item is again similar to the proof of Proposition~\ref{prop_saturated_is_open}:
   $\Sip_{r,X}$ is non-empty by 
   Lemma~\ref{lem_hilb_function_of_very_general_tuple}.
   Moreover, the condition defining $\Sip_{r,X}$ in the component constructed above is the intersection of countably many open conditions on the Hilbert function of the saturation and an additional one on the reducedness of $\zeroscheme(I)$. 
\end{proof}

\begin{notation}
  We will use the following notation motivated by Lemma~\ref{lem_hilb_function_of_very_general_tuple} and 
  Proposition~\ref{prop_motivation_for_h_r_X}:
  \begin{itemize}
    \item The function $h_{r,X}$ is the 
          \emph{generic Hilbert function of $r$ points on $X$}.
    \item Any ideal $I\in \Hilb^{h_{r,X}}_S$
          that belongs to the set $\Sip_{r,X}$  
          from Proposition~\ref{prop_motivation_for_h_r_X}
          (that is, $I$ is saturated and defines $r$ points of $X$)
          is called an \emph{``ip''} 
          (which stands for an \emph{ideal of points}, 
            implicitly, \emph{with a generic Hilbert function}).
    \item The abbreviation \emph{``sip''} in $\Sip_{r,X}$ stands 
            for \emph{set of ideals of points}.
    \item Any ideal in the (unique) irreducible component 
          of the multigraded Hilbert scheme $\Hilb^{h_{r,X}}_S$
          containing $\Sip_{r,X}$ is a \emph{lip} (\emph{limit of ideals of points}).
    \item The component containing $\Sip_{r,X}$ 
          is called the \emph{slip} (\emph{scheme of limits of ideals of points})
          and denoted $\Slip_{r,X}$. 
    \end{itemize}
\end{notation}

Thus $\Slip_{r,X}=\overline{\Sip_{r,X}}$ and it is an irreducible component of the multigraded Hilbert scheme. 
Very roughly, this component parameterises ideals of $r$-tuples of points in $X$ (ip) together with the limits of such ideals (lip).

The following is the analogue of the multigraded apolarity 
   (Proposition~\ref{prop_multigraded_apolarity}) for border rank.

\begin{thm}[Border apolarity]\label{thm_nonsaturated_apolarity}
Consider a smooth toric projective variety $X$ embedded in $\PP (H^0(L)^*)=\PP (\widetilde{S}_L)$.
   Suppose $F \in \widetilde{S}_L$ is a homogeneous polynomial of degree $L$.
   Then the border rank $\borderrank F$ is at most $r$ if and only if there exists a lip     $I \in \Slip_{r,X}$ such that $I \subset \apolar$. 
\end{thm}

The theorem is a corollary from the following statement.

\begin{lemma}\label{lem_strong_nonsaturated_apolarity}
   Fix a positive integer $r$ and a line bundle $L\in \Pic(X)$ and  set $r': = h_{r,X}(L)$.
   Let $\Slip_{r,X}$ be the irreducible component of the multigraded Hilbert scheme as above.
   Let $\sigma_r:=\sigma_r(X)\subset \PP(H^0(L)^*)$ be the secant variety of $X$  embedded via $L$.
   Denote by $Gr:= Gr\left(\PP^{r'-1}, \PP (H^0(L)^*)\right)$ the Grassmannian of projective linear subspaces $\PP^{r'-1}$ in  $\PP (H^0(L)^*)$.
   Then:
   \begin{itemize}
    \item The natural map $\rho \colon \Slip_{r,X} \to Gr$ 
            taking a homogeneous ideal $I$ to  $I_L^{\perp}\subset H^0(L)^*$ is regular.
    \item Define $\ccU \subset \Slip_{r,X} \times \PP(H^0(L)^*)$ to be the pullback via $\rho$ 
            of the universal subbundle:
            \[
               \ccU = \set{(I, [F]) \mid I \in \Slip_{r,X}, [F] \in \PP (H^0(L)^*) , [F] \in \rho(I)}.
            \]
           Then the secant variety $\sigma_r(X)$ is equal to the image of $\ccU$ under the projection $\ccU \to \PP (H^0(L)^*)$ on the second factor:
           \[
              \sigma_r(X) = \set{[F]\in \PP(H^0(L)^*) \mid \exists I \in \Slip_{r,X} \text{ such that } [F] \in \rho(I)}.
           \]
   \end{itemize}
\end{lemma}

We stress that there is no closure in the final equation of 
Lemma~\ref{lem_strong_nonsaturated_apolarity}.
We also note that $\ccU$ of the second item is a natural extension of 
the concept of abstract secant variety;
see for instance \cite[\S4.7.3]{landsberg_geometry_and_complexity},  \cite{chiantini_lectures_on_structure_of_projective_embeddings}, 
\cite[p.144]{harris} (where the abstract secant variety is called an 
   \emph{incidence correspondence}),
or \cite[V.1]{zak_tangents} (where it is denoted $S^{r-1}_{X}$).

\begin{prf}
  The natural map in the first item exists and is regular by the universal properties of
  the Grassmannian $Gr$ and of the multigraded Hilbert scheme.
    To prove the second item note that $\Slip_{r,X}$ is projective
       by \cite[Cor.~1.2]{haiman_sturmfels_multigraded_Hilb}, 
       thus $\ccU$ is projective and therefore the image of $\ccU$ under the projection is also closed in $\PP(H^0(L)^*)$.
    Moreover, 
       by Proposition~\ref{prop_motivation_for_h_r_X} 
       a very general lip $I\in \Slip_{r,X}$ 
       is the saturated ideal of $r$ distinct points 
       $\setfromto{p_1}{p_r} \subset X$.
    The fibre 
    $\PP^{r'-1} = \ccU_{I} \subset \set{I}\times \PPof{H^0(L)^*}$
       is the linear span 
       $\linspan{\fromto{p_1}{p_r}} = \PPof{I_L^{\perp}}$.
    That is, $\PP^{r'-1} \subset \sigma_r$, 
       and the image of $\ccU \to \PPof{H^0(L)^*}$
       is contained in $\sigma_r$.
    On the other hand, reversing the above argument, 
       we pick a very general point of 
       $[F]\in \sigma_r$.
    It is contained in the span of $r$ points 
       $\setfromto{p_1}{p_r} \subset X$ in very general position,
       and $I = \idealof(\setfromto{p_1}{p_r})$ 
       is a saturated ideal with Hilbert function $h_{r,X}$
       (Lemma~\ref{lem_hilb_function_of_very_general_tuple}).
    Thus $I\in \Slip_{r,X}$ and $[F]\in \PPof{I_L^{\perp}}$ 
       by the usual apolarity 
       (Proposition~\ref{prop_multigraded_apolarity}) 
       and therefore $[F]$ is in the image of 
       $\ccU\to \PPof{H^0(L)^*}$. 
    Therefore the image of $\ccU\to \PPof{H^0(L)^*}$
       is dense in $\sigma_r$.
 \end{prf}

Lemma~\ref{lem_strong_nonsaturated_apolarity} generalises analogous statements 
   that relate the Hilbert scheme and secant varieties (or cactus varieties) to high degree Veronese varieties
   (see \cite[Prop.~11]{bernardi_gimigliano_ida} or \cite[Prop.~2.5]{nisiabu_jabu_cactus}). 
Here we replace the Hilbert scheme by $\Slip_{r,X}$ and we avoid restrictions on $r$ and the embedding $X\subset \PP(H^0(L)^*)$.
The map $\rho$ is the analogue of the linear span of the zero locus $\linspan {\zeroscheme(I)}$
(see Remark~\ref{rem_interactions_between_all_sorts_of_things}), 
   and it agrees with the (scheme-theoretic) linear span for saturated ideals $I$.
In general, the span $\linspan{\zeroscheme(I)}$ is contained in the linear space $\rho (I)$.

The main advantage of this approach is the lack of ``closure'' in the expression for the secant variety. 
To some extent this is illusory, as we use closure to define the component $\Slip_{r,X}$. 
Nevertheless, an analogous approach turned out to be highly efficient in the setting of 
the Hilbert scheme and $X$ a projective space embedded via the Veronese map of degree $d$ 
with $r\leqslant d-1$.
Subsequent research shows this is also a useful method 
to estimate the border rank of points with large groups of symmetries, such as matrix multiplication tensors
(see~\cite{conner_harper_landsberg_border_apolarity_I}).

\begin{rmk}
   In all the interesting situations one can assume $r'=r$ in the statement of Lemma~\ref{lem_strong_nonsaturated_apolarity}.
   Otherwise, if $r'\ne r$, then $r' = \dim H^0(L)$, so $Gr = \set{\PP(H^0(L)^*)}$ is a single point and the secant variety $\sigma_r$
    is equal to $\PP (H^0(L)^*)$.
   Moreover, any point $[F]\in \PP(H^0(L)^*)$ is in the span of  any sufficiently generic configuration of $r$ points in $X$.
   In other words, any general enough $r$ points span $\PP(H^0(L)^*)$.
   Thus the situation $r' \ne r$ occurs only in very boring cases.
\end{rmk}

\begin{prf}[ of~Theorem~\ref{thm_nonsaturated_apolarity}]
  First suppose that $\borderrank F \leqslant r$. 
  In the setting of Lemma~\ref{lem_strong_nonsaturated_apolarity},
      pick $u \in \ccU$ such that $u = (I, [F])$ 
      with $I \in \Slip_{r, X}$.
  Let $E = \rho (I)$, so that 
      $[F] \in E \simeq \PP^{r'-1} \subset \PPof{H^0(L)^*} $.
  Then $E = \PPof{(I_L)^{\perp}}$, and 
   \[
     [F] \in \PPof{(I_L)^{\perp}} \subset \PPof{H^0(L)^*}.
   \]
   Equivalently, $I_L \subset (\apolar)_L \subset H^0(L)$. 
   By Proposition~\ref{prop_apolar_in_degree_deg_F_vs_apolar}  we must have
      $I \subset \apolar$ as claimed.
 
   Now suppose $I \in \Slip_{r, X}$ is such that $I \subset \apolar$.
   Since $\Sip_{r,X}$ is dense in $\Slip_{r,X}$ 
      in the analytic topology by Proposition~\ref{prop_motivation_for_h_r_X},
      there exists a sequence $I^k\in \Sip_{r,X}$
      such that $I^k \stackrel{k\to \infty}\to I$.
   Suppose that 
      $\zeroscheme(I^k)= \setfromto{p_1^k}{p_r^k} \subset X \subset \PPof{H^0(L)^*}$ 
      (we view $p_i^k$ as elements of $\PPof{H^0(L)^*}$).
   By the assumption $[F] \in \PP(I_{L}^{\perp}) = \rho(I)$, 
       and $\lim_{k\to \infty} \rho(I^k) = \rho (I)$.
   Thus
   \[
      [F] \in \rho(I)  = \lim_{k\to \infty} \rho\big(I^k\big) 
                       = \lim_{k\to \infty} \linspan {\zeroscheme\!\left(I^k\right)} 
= \lim_{k \to \infty} \linspan{\fromto{p_1^k}{p_r^k}}.
   \]
   Hence $F$ can be obtained as a limit of points of rank at most $r$ as claimed.
\end{prf}

\section{Ideals calculating border rank}\label{sec_VSP}

Suppose for a while that $X=\PP^n$ and $L= \ccO_{\PP^n}(d)$ so that $\varphi_{|L|}$ is the $d$-th Veronese embedding.
For a homogeneous polynomial 
   $F \in \widetilde{S}_L =H^0(L)^*\simeq S^d\CC^{n+1}$ 
   and an integer $r$, 
   the \emph{variety of sums of powers} $\mathit{VSP}(F, r)$ is defined as the closure 
     in the standard Hilbert scheme\footnote{For 
     readers not familiar with the notion of standard Hilbert scheme,
     we mention that here it is enough to think of the standard Hilbert scheme
     as a sufficiently nice compactification of the set of unordered tuples of $r$ distinct points of $X$.}
     $\usualHilb_X^r$
     of the set of $r$-tuples 
     $\setfromto{\left[\ell_1\right]}{\left[\ell_r\right]}$ of points in $\PP^n$ 
     such that $[F] \in \linspan{\fromto{\left[\varphi_{|L|}(\ell_1)\right]}{\left[\varphi_{|L|}(\ell_r)\right]}}$.
In other words, $\mathit{VSP}(F,r)$ is responsible for all the solutions to the decomposition problem for $F$ into $r$ simple summands 
   and $\mathit{VSP}(F,r) \ne \emptyset$ if and only if $r(F) \leqslant r$. $\mathit{VSP}$ has been intensively studied in 
   (for instance)
   \cite{ranestad_schreyer_VSP}, \cite{iliev_ranestad_K3_of_genus_8_and_VSP},
   \cite{ranestad_voisin_VSP_and_divisors_in_the_moduli_of_cubic_fourfolds}, 
   \cite{nisiabu_jabu_teitler_Waring_decompositions_of_monomials}.
More generally, for any smooth projective toric variety $X$, an analogue of $\mathit{VSP}$ 
   is considered in \cite{gallet_ranestad_villamizar}.
   
In this section we introduce a border version of $\mathit{VSP}$, 
  which is responsible for the set of solutions to border rank decompositions.

\subsection{Border VSP}

Back to the general situation, 
   let $X$ be a smooth toric projective variety,
   $L$ a very ample line bundle, $r\in \ZZ$, and $F\in H^0(L)^*$.
As recalled above, the $\mathit{VSP}$ is traditionally considered 
   as a subset of the standard Hilbert scheme,
   which  works fine for the rank decompositions, 
   but not so well when considering solutions to the border rank problem.
Instead, we propose to look at the analogue of $\mathit{VSP}$ inside the 
   multigraded Hilbert scheme, and more specifically, inside $\Slip_{r,X}$.
We define the \emph{border $\mathit{VSP}$}:
\[
  \bVSP (F, r):=\set{ I \in  \Slip_{r,X} \mid I \subset \apolar}.
\]
\begin{prop}\label{prop_basics_of_border_VSP}
   In the notation and assumptions above we have:
   \begin{itemize}
   \item $\bVSP (F, r) \subset \Slip_{r,X}$ is a Zariski closed subset, in particular,
     it has the structure of a projective (possibly reducible) variety.
     \item $\bVSP (F, r) \ne \emptyset \iff \borderrank{F} \leqslant r$,
     \item $\bVSP (F, r) \cap \Sip_{r,X} \ne \emptyset \Longrightarrow \rank{F} \leqslant r$.
   \end{itemize}
\end{prop}
\begin{proof}
   The closedness follows from 
   Lemma~\ref{lem_strong_nonsaturated_apolarity}, 
    because $\bVSP (F, r)$  is equal to the image 
    under the (projective) map $\ccU \to \Slip_{r,V}$
    of the preimage of the point  $[F]$ under the map 
    $\ccU \to \PP(H^0(L)^*)$.
   The border $\mathit{VSP}$ is non-empty if and only if $\borderrk(F) \leqslant r$ 
      by Theorem~\ref{thm_nonsaturated_apolarity}, 
      and the final item is clear. 
\end{proof}

The converse implication in the last item of 
Proposition~\ref{prop_basics_of_border_VSP} is false, as indicated 
in Example~\ref{ex_three_points_on_a_line} below. However, it is possible to reformulate the left hand side of this implication to make it into a necessary and sufficient condition.

In other words, we find the border $\mathit{VSP}$ a convenient expression 
  for the set of solutions to the approximate decomposition problem.
It would be hard to write all possible ways in which a 
   given polynomial $F\in \widetilde{S}_L$ can be approximated 
   using $r$ simple terms, as such a space would be infinite-dimensional.
Instead, the border $\mathit{VSP}$ expresses all possible limiting ideals 
   without bothering to write each ideal as a limit,
   and thus it gets rid of this infinite-dimensional 
   part of the problem.
Nevertheless the border $\mathit{VSP}$ has more information than, for instance, 
   just the limiting linear span 
   (which can be recovered from $\bVSP$ using the map $\rho$ from
     Lemma~\ref{lem_strong_nonsaturated_apolarity}).

\begin{example}\label{ex_three_points_on_a_line}
   Suppose $X=\PP^2$, $L=\ccO(d)$ for $d\geqslant 4$ and 
   $F=x_1^{(d)}+ x_2^{(d)} + (x_1 + x_2)^{(d)}$. 
   That is, $F$ depends only on two out of three variables, 
   $\rank{F} =\borderrank{F}=3$, and the above expression is the unique 
   (up to order) decomposition of $F$ into three simple summands.
   In this case, $\bVSP(F,3) = \set I $, 
      where $I \subset \CC[\alpha_0, \alpha_1, \alpha_2]$ is 
   $I=(\alpha_0^2, \alpha_0\alpha_1,\alpha_0\alpha_2, \alpha_1^2\alpha_2 - \alpha_1\alpha_2^2)$.
   Note that $I \notin \Sip_{3,X}$ because its saturation is equal to 
     $I^{\sat}=(\alpha_0, \alpha_1^2\alpha_2 - \alpha_1\alpha_2^2)$,
     which has the Hilbert function
     \[h(i) = 
\begin{cases}
                                                0 &\text{if }\   i<0,\\
                                                i+1 &  \text{if } -1 \leqslant i \leqslant 2, \\
                                                3  &\text{if }\  i \geqslant 2,                                     \end{cases}  \                                           
\]
   which we briefly write as $h = (1,2,3,3,\dotsc)$, whereas $h_{3,\PP^2} = (1,3,3,\dotsc)$.
\end{example}

There are a multitude of other examples where all elements of $\bVSP(F,r)$ are non-saturated.
The monomial $x_0^{(d)} x_1^{(2)} x_2$ (for $d\ge 3$) has border rank $6$
\cite[Thm~11.3]{landsberg_teitler_ranks_and_border_ranks_of_symm_tensors}, and $\bVSP\left(x_0^{(d)} x_1^{(2)} x_2,6\right)=\set{(\alpha_1^3, \alpha_0\alpha_2^2, \alpha_1\alpha_2^2,\alpha_2^3)}$. 
Similarly, if $\fromto{p_1}{p_6}\in \CC^3$ 
  are non-zero vectors whose projectivisations 
  are distinct and lie on a smooth conic curve, 
  and $F=p_1^{(d)}+ \dotsb + p_6^{(d)}$ 
  for $d \gg 0$,
  then also $\bVSP(F,6)$ consists of 
  a single non-saturated ideal.
Any wild polynomial or tensor $F$
  (see \cite[\S1]{nisiabu_jabu_smoothable_rank_example}, 
  and Subsection~\ref{sec_wild_cubic})
  also has only non-saturated ideals in  
  $\bVSP(F,\borderrank{F})$, by definition.

\subsection{Automorphism group action}
   
The action of the automorphism group $\Aut(X)$ on $X$ induces a natural action of $\Aut(X)$ on $\Slip_{r, X}$. 
Below we exploit an advantage of $\bVSP$ over the usual $\mathit{VSP}$: to define it we do not need to use closure and still we get a projective (in particular compact) variety. Thus the group action is largely determined by projective orbits, and in some cases fixed ideals 
are sufficient to study the solutions to border rank decompositions.

\begin{thm}[Fixed Ideal Theorem]\label{thm_G_invariant_bVSP}
   Suppose $F\in H^0(L)^*$ and $G \subset \Aut(X)$ is a subgroup that preserves $[F] \in \PPof{H^0(L^*)}$.
   Then $\bVSP(F,r)$ is a $G$-invariant subset of $\Slip_{r,V}$.
   In particular, 
      \begin{itemize}
       \item  $\br(F) \leqslant r$ 
      if and only if there exists a projective orbit of $G$
      contained in $\bVSP(F,r)$, 
      \item if $B\subset G$ is a connected solvable subgroup (for example, an algebraic torus, or a Borel subgroup), 
            then $\br(F) \leqslant r$ 
              if and only if there exists a $B$-fixed point 
              in $\bVSP(F,r)$.
      \end{itemize}
\end{thm}
A weaker version of this theorem
  appears as the Normal Form Lemma in 
  \cite[Lem.~3.1]{landsberg_michalek_geometry_of_border_rank_decompositions}.
\begin{prf}
   The $G$-invariance follows from the definition of $\bVSP(F,r)$,
      since $\apolar$ is $G$-invariant. 
   Alternatively, one can use the fact that the maps 
      $\rho \colon \ccU \to   \PPof{H^0(L^*)}$ 
      and $\ccU \to \Slip_{r,X}$ from 
      Lemma~\ref{lem_strong_nonsaturated_apolarity} 
      are $\Aut(X)$-equivariant 
      (in particular $G$-equivariant, 
      and thus $\bVSP(F,r)$ is $G$-invariant). 
 
   The first item follows from Proposition~\ref{prop_basics_of_border_VSP}, 
      because in particular $\bVSP(F,r)$ is non-empty if and only if it admits 
      a closed (hence projective) orbit.
   The second item follows from the Borel Fixed Point Theorem (or the Lie-Kolchin Theorem) \cite[Thm~III.10.4, Cor.~III.10.5]{borel}.
\end{prf}

\begin{rmk}
    Note that one cannot hope for similar statements for rank 
    (or even cactus or smoothable ranks), 
    as the set of solutions to such decompositions is not necessarily compact in any sense.
    In contrast, 
      \cite{derksen_teitler_lower_bounds_for_ranks_of_invariant_forms}
      gives lower bounds for rank, smoothable rank 
      and cactus rank (but not border rank) 
      of invariant polynomials using different methods 
      that also use apolarity.
\end{rmk}

Due to significance of the Fixed Ideal Theorem 
(or \emph{fit}; Theorem~\ref{thm_G_invariant_bVSP}), 
   we decrypt it to get rid of most of the notation.

\begin{cor} \label{cor_FIT_decrypted}
    Suppose $X\subset \PPof{H^0(L)^*}$ is a smooth 
      toric projective variety embedded into the projective space via the complete linear system of a very ample line bundle $L$, and suppose $S$ is the Cox ring of $X$.
    Assume $[F]\in \PPof{H^0(L)^*}$ is $B$-invariant 
       for a connected solvable group $B\subset \Aut(X)$.
    Then the border rank of $F$ is at most $r$ if and only if there exists a homogeneous ideal $I \subset S$ such that 
    \begin{enumerate}
     \item \label{item_FIT_I_fixed}
            $I$ is $B$-invariant, in particular each graded piece $I_D$ is 
               a $B$-invariant linear subspace of $S_D = H^0(D)$,
     \item  \label{item_FIT_I_subset_F_perp}
            $I  \subset \apolar$,
     \item  \label{item_FIT_Hilbert_function}
            $\dim \, (S/I)_D = \min (r, \dim S_D)$ for all $D\in \Pic(X)$, and 
     \item  \label{item_FIT_I_is_limit_of_pts}
            $I$ is a flat limit of saturated ideals of $r$ distinct points in $X$ which are in very general position. 
    \end{enumerate}
\end{cor}
   
Already conditions 
\ref{item_FIT_I_fixed}--\ref{item_FIT_Hilbert_function}
are often very restrictive: if the group of automorphisms of $F$ is large enough, then at least for small degrees $D$ there might be only finitely many subspaces fixed by $B$, and thus verifying these conditions
boils down to checking those finitely many cases. 
Condition~\ref{item_FIT_I_is_limit_of_pts} is more demanding to check. 
See Subsection~\ref{sec_smoothable_ideals}  for a brief discussion.

\section{Examples and applications}\label{sec_applications}

In this section we discuss three
  previously known examples and express them in terms of
  the border apolarity presented in this article.
Next we explain an application to characterise
   tensors of minimal border rank.

\subsection{Polynomials in the tangent space}

Suppose $X=\PP^n$ and $L=\ccO_{\PP^n}(d)$ for $d\geqslant 3$,
   so that $X\subset \PPof{\widetilde{S}_L}$ 
   is the $n$-dimensional $d$-th Veronese variety.
Let $F = x_0 ^{(d-1)}x_1$. 
Then it is well known that $\borderrank{F} =2$, while $\rank{F}=d>2$,
   and this is one of the first examples of this phenomenon discussed
   in textbooks. 
In our language, the lip (limit of ideals of points) $I\in \Slip_{r,X}$
   that arises from Theorem~\ref{thm_nonsaturated_apolarity},
   that is, $I\subset \apolar$ is
   \[
    I = (\alpha_1^2, \alpha_2,\dotsc, \alpha_n),
   \]
   and moreover it is the unique such ideal, 
   that is, $\bVSP(x_0 ^{(d-1)}x_1, 2) = \set{I}$.
   In particular, $I$ is saturated,
   and it also calculates the cactus and smoothable ranks 
     (see for instance 
      \cite{nisiabu_jabu_smoothable_rank_example} 
      or \cite{bernardi_brachat_mourrain_comparison}).
   
The reader will easily generalise this example 
   to border rank $2$ points over other toric varieties. 
However, the uniqueness of $I$ does not hold in general.
For instance, if $X=\PP^1$ and $L=\ccO_{\PP^1}(2)$, 
   or if $X=\PP^1\times \PP^1$ and $L=\ccO_{\PP^1\times \PP^1}(1, 1)$,
   then the lip calculating the border rank is not unique.

\subsection{A tensor of border rank \texorpdfstring{$3$}{3}}\label{sec_tensor_of_border_rank_3}

This example is based on case (iv) 
of \cite[Thm~1.2]{landsberg_jabu_third_secant}.
Suppose $X=\PP^2\times\PP^2\times \PP^2$ 
   and $L=\ccO_{\PP^2\times\PP^2\times \PP^2}(1,1,1)$.
   Using  the notation of Example~\ref{ex_Segre}, let 
   \newcommand{\shortplus}{\!+\!}\newcommand{\shortminus}{\!-\!}
\begin{alignat*}{1}
F &= x_0\otimes y_0 \otimes z_0 
  + x_1\otimes y_0 \otimes z_1 
  + x_1\otimes y_1 \otimes z_0
  + x_2\otimes y_0 \otimes z_2
  + x_2\otimes y_2 \otimes z_0\\
=&\! \lim_{t\to 0} 
    \tfrac{1}{t}\bigl(\!(t x_0 \shortminus x_1\shortminus x_2)\otimes y_0 \otimes z_0 +x_1\otimes(y_0\shortplus t y_1)\otimes(z_0\shortplus t z_1) 
                 + x_2\otimes(y_0\shortplus t y_2)\otimes(z_0\shortplus t z_2)\!\bigr).
\end{alignat*}
This tensor has border rank $3$. 
The expression of $F$ as a limit enables an explicit calculation 
  of the ideal $I \in \Slip_{3,X}$ such that $I\subset \apolar$:
\[ I =  \left(
     \begin{array}{ccc}
       \alpha_0^2,\  \alpha_0\alpha_1,  \ \alpha_0\alpha_2,\\
      \beta_1^2,\
      \beta_1\beta_2,\
      \beta_2^2,\\
      \gamma_1^2,\
      \gamma_1\gamma_2,\
      \gamma_2^2,\\
       \alpha_1\alpha_2(\alpha_1 - \alpha_2),\\
       \alpha_0\beta_0-\alpha_1\beta_1,\ 
       \alpha_0\beta_0-\alpha_2\beta_2,\ 
       \alpha_0\beta_1,\ 
       \alpha_0\beta_2,\ 
      \alpha_1\beta_2,\ 
      \alpha_2\beta_1,\\
      \alpha_0\gamma_0-\alpha_1\gamma_1,\ 
      \alpha_0\gamma_0-\alpha_2\gamma_2,\ 
      \alpha_0\gamma_1,\
      \alpha_0\gamma_2,\
      \alpha_1\gamma_2,\
      \alpha_2\gamma_1,\\
      \beta_0\gamma_1 - \beta_1\gamma_0,\
      \beta_0\gamma_2 - \beta_2\gamma_0,\      
      \beta_1\gamma_1,\
      \beta_1\gamma_2,\
      \beta_2\gamma_1,\
      \beta_2\gamma_2\\
   \end{array}\right).
   \]
Thus $I$ has $28$ minimal generators:
   three in each of the multidegrees 
   $(2,0,0)$, $(0,2,0)$ and $(0,0,2)$,
one in multidegree $(3,0,0)$, and
six in each of the multidegrees $(1,1,0)$, $(1,0,1)$, and $(0,1,1)$.
This ideal is \emph{not} saturated: 
   for example, $\alpha_0\notin I$, 
   but $\alpha_0\cdot (\alpha_0,\alpha_1,\alpha_2) \subset I$,
   hence $\alpha_0$ is in the saturation of $I$.
The zero set of $I$ in $\PP^2\times \PP^2\times \PP^2$ 
   is three distinct points 
   $\set{[x_1],[x_2],[x_1+x_2]}\times [y_0]\times[z_0]$.
   
In this case $\bVSP(F, 3) \ne \set{I}$. 
This can be recovered from the proof 
   of \cite[Thm~1.11]{landsberg_jabu_third_secant},
   where it is shown that the same $F$ can be obtained 
   by making three points converge to a single point, 
   and thus the limiting ideal could also be an ideal 
   that has only one point of support.
It would be an interesting follow-up project 
   to determine if there is any interesting
   geometry in the variety (or scheme) $\bVSP(F,3)$,
   analogously to other $\mathit{VSP}$'s studied for instance in
   \cite{gallet_ranestad_villamizar}.

\subsection{Wild cubic in five variables}\label{sec_wild_cubic}

Let $X=\PP^4$ and $L=\ccO_{\PP^4}(3)$.
In \cite{nisiabu_jabu_smoothable_rank_example} we studied
  the following cubic polynomial:
  \begin{align*}
     F &= x_0^{(2)} x_2 - (x_0+x_1)^{(2)} x_3 + x_1^{(2)} x_4\\
       &= \lim_{t\to 0} \frac{1}{t}\Bigl( (x_0 + t x_2)^{(3)} - 
                        (x_0 + x_1 +  t x_3)^{(3)} 
                        - \tfrac{1}{4} (2x_1 - t x_4)^{(3)} \\
                       &\phantom{=\lim \frac{1}{t} \Bigl(} - \tfrac{1}{3} ( x_0 -  x_1)^{(3)} 
                        + \tfrac{1}{3} ( x_0 + 2x_1)^{(3)}\Bigr).
  \end{align*}
The results of \cite{nisiabu_jabu_smoothable_rank_example} 
   show that this is an example of a ``wild'' polynomial in 
   the language of \cite[\S1]{nisiabu_jabu_smoothable_rank_example},
   and this implies that its border rank is more difficult to analyse.
In fact, the example discussed in 
   Subsection~\ref{sec_tensor_of_border_rank_3} is also ``wild'' in this sense (see \cite[\S2.3]{nisiabu_jabu_smoothable_rank_example}). 
In the context of the present article, 
   ``wildness'' of $F$ can be phrased as
   the fact that there is no saturated ideal in $\bVSP(F, \borderrank{F})$.
   
\begin{rmk}
   The cubic polynomial $F$ is concise 
      (strictly depends on all variables), 
      but has vanishing hessian
\cite[Example~7.1.5]{russo_geometry_of_special_varieties}.
   Thus it would be interesting to investigate further the relation between such special polynomials (or hypersurfaces) and wild examples of polynomials.
We thank Giorgio Ottaviani for this remark.
\end{rmk}

Here we present an explicit expression for 
   the lip $I\in \bVSP(F,5)$,
   arising from the presentation of $F$ as a limit:
   \begin{align*}
    I:=  \biggl(&
   \alpha_0\alpha_2 + \alpha_1\alpha_3 + \alpha_1\alpha_4, \
    \alpha_0\alpha_3 - \alpha_1\alpha_3, \
    \alpha_0\alpha_4, \
    \alpha_1\alpha_2, \
    \alpha_2^2, \
    \alpha_2\alpha_3, \
    \alpha_2\alpha_4, \
    \alpha_3^2, \
    \alpha_3\alpha_4, \
    \alpha_4^2,\\
   & \alpha_0^4\alpha_1 - \half \alpha_0^3\alpha_1^2 
        - \alpha_0^2\alpha_1^3 + \half\alpha_0\alpha_1^4
\biggr).
\end{align*}
This ideal is thus generated by ten quadrics coinciding with the ten quadrics in $\apolar$, and also a quintic.
Moreover $I$ is not saturated: 

\[
  I^{\sat} = \left(\alpha_2,\alpha_3, \alpha_4,\alpha_0^4\alpha_1 - \half \alpha_0^3\alpha_1^2 - \alpha_0^2\alpha_1^3 + \half\alpha_0\alpha_1^4 \right).
\]
The ideal $I$ is not a unique member of $\bVSP(F, 5)$
   and it would be interesting to understand the geometry 
   of $\bVSP(F, 5)$.\footnote{%
   Recently, \cite[Thm~6.6]{huang_michalek_ventura_wild_forms} has solved this problem, as a special case of a much more general discussion.}

\subsection{Tensors of minimal border rank}

We commence by recalling the definition of conciseness.
\begin{defin}\label{def_concise}
   Suppose $S = \CC[\fromto{\alpha_1}{\alpha_{n+w}}]$
     is the Cox ring of $X$ and 
     $F\in \widetilde{S}=\CC[\fromto{x_1}{x_{n+w}}]$ 
     is homogeneous.
   We say $F$ is \emph{concise} if 
   $\apolar_{\deg{\alpha_i}} = 0$ for all $i$.
\end{defin}

This definition coincides with the standard notion of conciseness 
   for tensors and polynomials:
\begin{itemize}
 \item if $F$ is concise, then $F$ strictly depends 
         on all variables $x_i$, even after any automorphism of $X$, and
 \item if 
       $X = \PP^{a_1}\times \dotsb \times \PP^{a_w}$,
       then the converse of the above also holds:
       if $F$ strictly depends on all the variables,
          after any automorphism of $X$, then $F$ is concise.
\end{itemize}

In particular, the following property is standard.
Suppose 
      $X = \PP^{a_1}\times  \dotsb \times \PP^{a_w}$
      and $F\in \widetilde{S}_{\ccO_X(d_1,\dotsc,d_w)}$ 
      is concise.
    Then 
      $\borderrank{F} \geqslant \max\setfromto{a_1+1}{a_w+1}$.
More generally:
\begin{prop}
    For any smooth toric projective variety, in the setting of Definition~\ref{def_concise}, 
    if $F$ is concise, then 
       $\borderrank{F} \geqslant \max\set{\dim H^0(\deg(\alpha_i)) \mid i \in \setfromto{1}{n+w}}$.
\end{prop}
This statement is a special case of the catalecticant bound
\cite[Cor.~5.5]{galazka_mgr} for border rank.

In the setting of the proposition  
   we say that $F$ \emph{has minimal border rank} 
   if $F$ is concise and 
   the border rank is equal to the minimal value from the proposition:
   \[
     \borderrank{F} 
     = \max\set{\dim H^0(\deg(\alpha_i)) \mid i \in \setfromto{1}{n+w}}.
   \]
   
A consequence of 
   \cite[Thm~4.8]{nisiabu_jabu_kleppe_teitler_direct_sums} 
   is that for $X=\PP^n$, if $F\in \widetilde{S}_{\ccO(d)}$ 
   has minimal border rank, then the apolar ideal $\apolar \subset S$ 
   has at least $n=\borderrank{F}-1$ minimal generators in degree $d$.
Still in the case of $X=\PP^n$ this is equivalent to the following claim:
\[
  \dim \biggl( S_{d} \left/ 
         \left(\apolar_{d-1} \cdot S_{1} \right)\right.\biggr) \geqslant n+1= \borderrank{F}.
\]
Here we show a generalisation of these claims to toric varieties.
First we present a version for the product of projective spaces, 
   where the dimensions of the factors are equal.
\begin{thm}\label{thm_minimal_border_rank_product_of_Pa_s}
   Suppose $X= \left(\PP^a\right)^{\times w}$ for an integer $a$,
     and $F\in \widetilde{S}_L$ 
     is a partially symmetric tensor of minimal border rank. 
   Then the apolar ideal $\apolar \subset S$ has at least 
     $a=\borderrank{F}-1$ minimal generators in degree $L$.
\end{thm}

This statement applies, for instance, to the case 
   $X=\PP^3\times \PP^3 \times \PP^3$ and its $4$-th secant variety.
The famous Salmon Problem posed by Allman 
   asked for the description of the ideal of this secant variety in $\PP^{63}$.
This was partially solved 
   by Friedland \cite{friedland_salmon_problem_paper},
   who gave its set-theoretic equations, 
   at the same time providing criteria for a tensor in 
   $\CC^4\otimes \CC^4\otimes \CC^4$ to have border rank $4$.
Theorem~\ref{thm_minimal_border_rank_product_of_Pa_s} 
   (with $a=4$, $w=3$) 
   provides some numerical necessary conditions for such tensors. 
It is an interesting problem to determine if (in higher dimensions)
   they are covered by previous research 
   and if they are sufficient for $a=4$ and $w=3$. 
   
More generally, for any toric variety we have 
  a slightly weaker version of this statement.

\begin{thm}\label{thm_minimal_border_rank_general_case}
   Suppose $X$ is a smooth toric projective variety
     and $F\in \widetilde{S}_L$ 
     is of minimal border rank.
   Let $i\in \setfromto{1}{n+w}$ be an integer 
     such that 
     $\borderrank{F}= \dim H^0(\deg(\alpha_i))$.
   Then  
   \[
       \dim \biggl( S_L \left/ 
         \left(\apolar_{L-\deg(\alpha_i)} \cdot S_{\deg{\alpha_i}} \right)\right.\biggr) \geqslant \borderrank{F}.
   \]
\end{thm}

The proofs of both theorems follow the idea that we phrase as the following lemma.

\begin{lemma}\label{lem_minimal_border_rank}
   Suppose $F$ is concise and 
      of minimal border rank,
      let $I \in \bVSP(F,\borderrank{F})$ and let $i$ be such that
      $\borderrank{F}= \dim H^0(\deg(\alpha_i))$.
   Then 
      \[
         \apolar_{L-\deg(\alpha_i)} \cdot S_{\deg{\alpha_i}} 
          \subset I_L.
      \]
\end{lemma}
\begin{prf}
   By the symmetry of the Hilbert function \cite[Prop.~4.5]{galazka_mgr}
      we have
      \begin{align*}
         \borderrank{F} =\dim S_{\deg(\alpha_i)} 
         &= \dim \left(S/\apolar\right)_{\deg(\alpha_i)}\\
         &= \dim \left(S/\apolar\right)_{L-\deg(\alpha_i)}
         = \dim \left(S/I\right)_{L-\deg(\alpha_i)}.
      \end{align*}
  Since $I_{L-\deg(\alpha_i)} \subset  \apolar_{L-\deg(\alpha_i)}$  
     by the definition of $\bVSP(F,\borderrank{F})$, 
     and their dimensions agree by the above calculation, 
     we must have $\apolar_{L-\deg(\alpha_i)}= I_{L-\deg(\alpha_i)}$.
  Thus
     $\apolar_{L-\deg(\alpha_i)} \cdot S_{\deg{\alpha_i}} \subset I_L$
     as claimed.
\end{prf}

\begin{prf}[ of Theorem~\ref{thm_minimal_border_rank_product_of_Pa_s}]
    Let $I\in \bVSP(F, a+1)$.
    Since every $i \in \setfromto{1}{(a+1)w}$ is such that
    $\borderrank{F}= \dim H^0(\deg(\alpha_i))$, by 
    Lemma~\ref{lem_minimal_border_rank} 
    we have
    \[
       \sum_{i} \apolar_{L-\deg(\alpha_i)} \cdot S_{\deg{\alpha_i}} 
          \subset I_L \subset \apolar_L\subset S_L.
    \]
    Therefore, the codimension of $\sum_{i} \apolar_{L-\deg(\alpha_i)} \cdot S_{\deg{\alpha_i}}$ in $S_L$ is at least
    the codimension of $I_L$ in $S_L$, that is, $\borderrank{F}$.
    On the other hand, the codimension of $\apolar_L$ in $S_L$
      is $1$.
    Therefore the codimension of 
    $\sum_{i} \apolar_{L-\deg(\alpha_i)} \cdot S_{\deg{\alpha_i}} \subset \apolar_L$ is at least $\borderrank{F}-1$,
      that is, we must have at least $\borderrank{F}-1$ minimal generators 
      of $\apolar$ in degree $L$. 
\end{prf}

\begin{prf}[ of Theorem~\ref{thm_minimal_border_rank_general_case}]
    By Lemma~\ref{lem_minimal_border_rank} 
      the codimension of $\apolar_{L-\deg(\alpha_i)} \cdot S_{\deg{\alpha_i}}$ in $S_L$ is at least $\codim I_L = \borderrank{F}$,
      and the claim follows. 
\end{prf}

\section{Concerning monomials}\label{sec_monomials}

Throughout this section we 
suppose 
$F= x_1^{(a_1)}x_2^{(a_2)} \dotsm x_{n+w}^{(a_{n+w})}
   \in \widetilde{S}_L$ 
       is a monomial.
The main interest is in the case of $X=\PP^n$, 
and our main goal in this section is to prove 
Theorem~\ref{thm_monomials_intro} which includes the calculation of the border rank of monomials in three variables, that is, for $X =\PP^2$. Along the way we prove some other statements of interest and show examples of applications of the border apolarity. 

\subsection{The upper bound} 

For monomials, the main case under consideration is 
  $X =\PP^n$.
Both the rank $\rank[\PP^n]{F}$ and the variety of sums of powers $\mathit{VSP}(F,\rank[\PP^n]{F})$ 
       are calculated in
\cite{carlini_catalisano_geramita_Waring_rank_of_monomials}
       and in
\cite{nisiabu_jabu_teitler_Waring_decompositions_of_monomials}.
Moreover, the border rank 
    $\borderrank[\PP^n]{F}$ is discussed in 
    \cite{oeding_border_rank_monomials},
     but at the time of submission of this article there are gaps in the argument.
For some other low dimensional toric varieties $X$, 
   the rank and border rank (and also cactus rank) of some monomials are calculated and estimated 
   in \cite{galazka_mgr}.
   
It is generally expected that in case $X=\PP^n$ 
   the border rank is equal to the well known upper bound \cite[Thm~11.2]{landsberg_teitler_ranks_and_border_ranks_of_symm_tensors}:
   \begin{equation}\label{equ_upper_bound_monos}
      \borderrank[\PP^n]{x_0^{(a_0)}x_1^{(a_1)} \dotsm x_{n}^{(a_{n})}}\leqslant (a_1+1)(a_2+1)\dotsm (a_n+1) 
      \text{ where } a_0\geqslant a_1\ge\dotsb\geqslant a_n.
   \end{equation}
If $n=1$, then equality in \eqref{equ_upper_bound_monos}
is standard and well known: for instance, see  \cite[Thm~11.2]{landsberg_teitler_ranks_and_border_ranks_of_symm_tensors} again.

   The upper bound
     generalises 
     to any smooth toric projective variety. 
   For $X= \PP^{a}\times \PP^{b}\times \PP^{c} \times \dotsm$
     this can be seen directly 
     from the submultiplicativity of rank and border rank: 
     if $X$ and $Y$ are toric varieties with Cox rings 
     $S[X]$ and $S[Y]$, 
     and $F\in \widetilde{S}[X]$, $G\in \widetilde{S}[Y]$ are homogeneous, then $S[X\times Y] = S[X] \otimes S[Y]$ and
     $\borderrank[X\times Y]{F \cdot  G} \leqslant \borderrank[X]{F}\borderrank[Y]{G}$.
   In general, the proof is a straightforward generalisation 
     of the argument for $\PP^n$, 
     as in~\cite{ranestad_schreyer_on_the_rank_of_a_symmetric_form}.
   \begin{lemma}\label{lem_monomial_upper_bound_on_general_toric_variety}
      Let $F=x_1^{(a_1)} \dotsm x_{n+w}^{(a_{n+w})}$ be a monomial 
        in the Cox ring $S$.
      Suppose $\AAA^n \simeq U\subset X$ is an open affine torus-invariant subset given as the complement of 
      $\zeroscheme(\alpha_{i_1}\dotsm\alpha_{i_w})$
      for some choice of pairwise different indices 
      $\fromto{i_1}{i_w}$. 
      Denote 
      \[
        J:= \setfromto{1}{n+w} \setminus\setfromto{i_1}{i_w},
      \]
        the complement of those indices.
      Then 
      \[
           \borderrank{F} \leqslant 
       \prod_{j\in J}(a_{j}+1).
      \]
   \end{lemma}
   \begin{prf}
      Consider the scheme 
      $R = \zeroscheme\bigl(\alpha_j^{a_{j} +1} \mid  j \in J\bigr)$.
      Note that its support is the torus-fixed point of $U$,
        and its length is $\prod_{j\in J}(a_{j}+1)$.
      Moreover, the ideal used to define $R$ is saturated 
        and thus $\idealof(R) \subset \apolar$.
      Also $R$ is smoothable by 
        \cite[Prop.~4.15]{cartwright_erman_velasco_viray_Hilb8}.
      Therefore $R$ shows that the smoothable rank of $F$ 
        is at most $\prod_{j\in J}(a_{j}+1)$, 
        and the smoothable rank is an upper bound for the border rank. 
      See for instance 
        \cite[Sect.~1 and~\S2.1]{nisiabu_jabu_smoothable_rank_example}
        for the definition and basic properties of smoothable rank, 
        including the above inequality comparing it with border rank. 
      The cases of $X=\PP^n$ are also discussed 
        in~\cite{ranestad_schreyer_on_the_rank_of_a_symmetric_form} 
        and~\cite{bernardi_brachat_mourrain_comparison}.
   \end{prf}

\subsection{Move-fit}
   
The following is an immediate consequence of the Fixed Ideal Theorem, Theorem~\ref{thm_G_invariant_bVSP}.
Compare also with Corollary~\ref{cor_FIT_decrypted}.
\begin{cor}[Monomial Version of Fixed Ideal Theorem, or move-fit]\label{cor_monomials}
For any smooth toric projective variety $X$,
   the border rank of a monomial 
     $F=x_1^{(a_1)} \dotsm x_{n+w}^{(a_{n+w})}$
     is at most $r$ 
   if and only if there exists a homogeneous ideal 
   $I \subset S=\CC[\fromto{\alpha_1}{\alpha_{n+w}}]$
   such that
   \begin{enumerate}
     \item \label{item_brk_of_monomials_I_is_monomial}
            $I$ is a monomial ideal,
     \item  \label{item_brk_of_monomials_I_subset_F_perp}
            $I  \subset \linspan{\fromto{\alpha_1^{a_1+1}}{\alpha_{n+w}^{a_{n+w}+1}}}$, 
     \item  \label{item_brk_of_monomials_Hilbert_function}
            $\dim (S/I)_D = \min (\dim S_D, r)$, and
     \item  \label{item_brk_of_monomials_I_is_limit_of_pts}
            $I$ is a flat limit of saturated ideals of $r$ distinct points in $X$ which are in a very general position. 
   \end{enumerate}
\end{cor}

We expect the last 
   item~\ref{item_brk_of_monomials_I_is_limit_of_pts}
   is redundant and it is implied by \ref{item_brk_of_monomials_I_is_monomial} and \ref{item_brk_of_monomials_Hilbert_function}, but we have little evidence for that, 
  except the analogy to smoothability of monomial ideals
  \cite[Prop.~4.15]{cartwright_erman_velasco_viray_Hilb8}.

Most of the examples below seem to be new cases of calculation 
  of border rank of monomials.
The only exception is Example~\ref{ex_monomials_monotonic_border_rank},
  which we include here as an illustration, despite existence 
  of a more elementary proof, which is certainly known to experts.
Example~\ref{ex_mono_222} is also covered by a more general 
   Example~\ref{ex_monomials_P2}.

\begin{example}\label{ex_mono_222}
   Suppose $X=\PP^2$ and $F=(x_0 x_1 x_2)^{(2)}$.
   Then $\borderrank{F} \leqslant 9$ by 
      \eqref{equ_upper_bound_monos}.
   If $\borderrank{F} \leqslant 8$, then by \emph{move-fit} 
      there exists a monomial ideal $I\subset (\alpha_0^3,\alpha_1^3, \alpha_2^3)$ with $\dim I_3 =2$, 
      $\dim I_4=7$ and $\dim I_5= 13$.
   The first condition means that $I_3$ (up to reordering the variables) contains $\alpha_0^3$ and $\alpha_1^3$. 
   The condition on $I_4$ means that we need one more monomial in $I_4$ other than those generated by $I_3$, say $\alpha_2^3\cdot\alpha_i \in I_4$.
   But then $\dim I_5\geqslant 15$, a contradiction, thus $\borderrank{F} = 9$.
\end{example}

\begin{example}\label{ex_mono_2211}
   Suppose $X=\PP^3$ and $F=x_0^{(2)} x_1^{(2)} x_2 x_3$.
   We claim $\borderrank{F} = 12$, which is equal to the 
      upper bound of \eqref{equ_upper_bound_monos}.
   Suppose for contradiction that $\borderrank{F} \leqslant  11$,
      and let 
      $I\subset \apolar = (\alpha_0^3,\alpha_1^3, \alpha_2^2, \alpha_3^2)$
      be the monomial ideal obtained 
      by \emph{move-fit}.
   Note that $\dim \apolar_3 =10$ and $\dim I_3=9$,
      thus all but one monomial from $\apolar_3$ 
      are in $I_3$.
   We have $\dim I_4 =24$ and $\dim I_5= 45$.
   \begin{itemize}
    \item  If $\alpha_3^2\alpha_i\notin I$ 
            for some  $i\ne 3$, 
            then $I$ contains powers of all variables, 
            hence its Hilbert polynomial is $0$, a contradiction.
    \item  If $\alpha_3^3 \notin I$,
             then 
             \[
               I_3 = (\alpha_0^3, 
             \alpha_1^3, \alpha_2^2, 
             \alpha_3^2\alpha_0,
             \alpha_3^2\alpha_1,
             \alpha_3^2\alpha_2)_3
             \]
             and $\dim I_4 \geqslant 26$, a contradiction.
      \item If $\alpha_0^3\notin I$, that is, 
            \[
               I_3 = (\alpha_1^3, \alpha_2^2, 
             \alpha_3^2)_3,
            \]
            then $\dim (I_3 \cdot S_1) =23$, 
            and in degree $4$ we need one more generator of $I$ of the form 
            $\alpha_0^3\alpha_i$.
            Independent of $i$, we have $\dim I_4\cdot S_1 \geqslant 47$, a contradiction. 
   \end{itemize}
   The above items cover all the possible choices of one monomial in $\apolar$,
   up to swapping $\alpha_0$ with $\alpha_1$ 
   or $\alpha_2$ with $\alpha_3$.
   Thus the claim about the border rank of $F$ is proved.
\end{example}

\begin{example}
   The methods illustrated in Examples~\ref{ex_mono_222} and \ref{ex_mono_2211} can be automated, also for other low-dimensional toric varieties.
   With a naive implementation, which among 
      all monomial subideals of $\apolar$ just searches for those with correct Hilbert function, 
      we have been able to check that the following monomials in four or five variables have the border rank predicted by the upper bound of \eqref{equ_upper_bound_monos}:
   \begin{align*}
         &  x_0^{(3)} x_1^{(3)} x_2 x_3 x_4, \quad
           x_0^{(2)} x_1^{(2)} x_2 x_3 x_4, \quad
           x_0 x_1 x_2 x_3 x_4,  
           \text{ and }\\
       &    x_0^{(a_1)} x_1^{(a_1)} x_2^{(a_2)} x_3^{(a_3)} \text{ for $3 \geqslant a_1$ and $2\geqslant a_2\geqslant a_3$.}
   \end{align*}
   The verification of the first example took 
      approximately 1 hour and 30 minutes, 
      while the smaller examples took up to a couple of minutes.
   The larger cases we tried exhausted the memory of 
      the machine we worked with, 
      but we did not try to make 
      the algorithm efficient in any way.
\end{example}

\begin{example}\label{ex_monomials_monotonic_border_rank}
   The border rank of monomials is monotonic in the exponents. 
   That is (for any smooth toric variety $X$), if 
   $F= x_1^{(a_1)} \dotsm x_{n+w}^{(a_{n+w})}$
    and 
   $G= x_1^{(b_1)} \dotsm x_{n+w}^{(b_{n+w})}$
   with $a_i \leqslant b_i$ for all $i$, 
   then $\borderrank{F}\leqslant \borderrank{G}$.
   There is a straightforward and elementary way to verify this, 
   but we present an argument involving the border apolarity 
   method as an illustration.
   Indeed, $\apolar[G] \subset \apolar$,
   and thus any ideal $I$ coming from move-fit for $G$
     works for $F$ as well.
   In particular, it is enough to verify whether we have equality in \eqref{equ_upper_bound_monos} for monomials with $a_0 = a_1$.
\end{example}

If $X=\PP^n$, Landsberg and Teitler in 
\cite[Thm~11.3]{landsberg_teitler_ranks_and_border_ranks_of_symm_tensors}
show that the border rank is equal to the upper bound of \eqref{equ_upper_bound_monos} for unbalanced monomials, 
that is, for $F=x_0^{(a_0)} \dotsm x_{n}^{(a_{n})}$ 
  with $a_0 \geqslant a_1 + \dotsb + a_n$.
We treat the next case (a slightly less unbalanced monomial) in the following example.

\begin{example}\label{ex_almost_unbalanced}
   Suppose $X=\PP^n$,
     $F=x_0^{(a_0)} \dotsm x_{n}^{(a_{n})}$,
     and $a_0 + 1 = a_1 + \dotsb + a_n$ with $a_i>0$.
   Then the border rank of $F$ is equal to 
     $r=(a_1+1)\dotsm(a_n+1)$.
   Indeed, in this case 
     $\codim (\apolar_{a_0+1} \subset S_{a_0+1}) = r-1$,
     thus if $\borderrank{F} \leqslant r-1$, 
     then the monomial ideal $I$ from the move-fit 
     satisfies $I_{a_0+1} = \apolar_{a_0+1}$.
     But $\apolar_{a_0+1}$ contains powers of all variables, thus the Hilbert polynomial of $I$ is equal to $0$, a contradiction.
\end{example}

\subsection{Macaulay ideal growth}   
  
   We will use monomial ideals contained in the apolar ideal 
     of a monomial. 
   In sufficiently low degrees the apolar ideal of a monomial has 
     a disjointness property,
     that is, the parts coming from different generators 
     are linearly independent, and the monomial subideal 
     will split accordingly.
   Therefore, we will be interested in 
     minimising the growth 
     of monomial submodules of a graded free module.
   In the case of a free module with one generator over the standard graded polynomial ring, 
     the answer is provided by a classical theorem of Macaulay.
   We need to recall the following definition.
   
   \begin{defin}
      Suppose $X=\PP^n$, $S=\CC[\fromto{\alpha_0}{\alpha_n}]$, 
        and $d$ is a positive integer.
      Then by \emph{lex-segment} in degree $d$ of colength $r$ we mean the linear subspace $\Lex_d^r \subset S_d$ 
        of codimension $r$ spanned by 
        the last $\dim S_d-r$ monomials of $S_d$ 
        in the \grevlex{} (reverse degree-lexicographical) order.
   \end{defin}

For a fixed positive integer $d$, 
   any non-negative integer $r$ can be uniquely written as
   \[
      r= \tbinom{a_d}{d} + \tbinom{a_{d-1}}{d-1} + \dotsb + \tbinom{a_1}{1},
   \]
   where $a_i$ are integers (called the Macaulay coefficients) 
   such that 
   $a_d> a_{d-1} > \dotsb > a_1 \geqslant 0$ 
   \cite[Lem.~4.2.6]{bruns_herzog_Cohen_Macaulay_rings}.
The Macaulay exponent is
\[
   r^{\langle d\rangle}:= \tbinom{a_d+1}{d+1} + \tbinom{a_{d-1}+1}{d} + \dotsb + \tbinom{a_1+1}{2}.
\]   
   \begin{lemma}[{Macaulay, 
       \cite[Prop.~4.2.8, Cor.~4.2.9]{bruns_herzog_Cohen_Macaulay_rings}
       or \cite[Thm~3.3, Prop.~3.7]{green_generic_initial_ideals}}]
      \label{lem_Macaulay_bound}
      Suppose $X=\PP^n$, and $S=\CC[\fromto{\alpha_0}{\alpha_n}]$, 
        and for some positive integer $d$, 
        a linear subspace $I \subset S_d$ 
        has codimension $r$. 
      Then 
      \[
        \codim \left(I\cdot S_1 \subset S_{d+1}\right) 
        \leqslant  \codim \left(\Lex_d^r \cdot S_1 \subset S_{d+1}\right)
        = r^{\langle d \rangle}.
      \]
   \end{lemma}

\begin{rmk}\label{rem_growth_of_algebras_vs_growth_of_ideals}
  The Macaulay Lemma is referred to as
    giving the \emph{maximal} possible 
    growth of a standard graded \emph{algebra} $S/(I)$. 
  Note that this upper bound does not depend on the number 
    of generators (variables) of the algebra,
    but it strongly depends on the degree. 
  From a complementary point of view,
    we can reinterpret this statement as giving the  \emph{minimal}
    growth of homogeneous \emph{ideals} in the polynomial ring.
  With this approach, one can easily notice that 
    the growth of an ideal strictly depends on 
    the number of variables,
    but it does not really depend on the degree we look at, 
    only on the shape of the ideal inside the polynomial ring.
  For instance, the growth of $I\subset S_e$ from degree $e$ to $e+1$ is the same as the growth of 
  $\alpha_0^{d-e}\cdot I \subset S_d$ from degree $d$ to $d+1$.
  Moreover $I$ is a lex-segment 
     if and only if $\alpha_0^{d-e}\cdot I$ is.
\end{rmk}

We mention two properties of the Macaulay exponent
  $r^{\langle d \rangle}$ that we will use.
\begin{lemma}\label{lem_moving_ideals_with_lots_of_entries}
  Suppose $d$, $e$, $q$, $r$ are non-negative integers,
    and $d \geqslant e > 0$.
  Then $q^{\langle d \rangle} + r^{\langle e\rangle} 
        \leqslant (q+r)^{\langle e \rangle}$.
\end{lemma}
\begin{prf}
  We first prove the lemma for $d=e$.
  Consider two independent polynomial rings 
    $\CC[\fromto{\alpha_1}{\alpha_k}]$ 
    and $\CC[\fromto{\beta_1}{\beta_l}]$,
    both having sufficiently many variables.
  Choose two subspaces
  $I_{\alpha}=\Lex_e^q \subset \CC[\fromto{\alpha_1}{\alpha_k}]_e$
  and  
  $I_{\beta}= \Lex_e^r \subset \CC[\fromto{\beta_1}{\beta_l}]_e$.
  Finally, let 
  \[
    I = (I_{\alpha}) + (I_{\beta}) + (\alpha_i\beta_j \mid 1 \leqslant i \leqslant k, 1\leqslant j \leqslant l) \subset  \CC[\fromto{\alpha_1}{\alpha_k},\fromto{\beta_1}{\beta_l}].
  \]
   Setting $A_{\alpha}$, $A_{\beta}$ and $A$ to be the quotient algebras by $(I_{\alpha})$, $(I_{\beta})$ and $I$ respectively,
   we have
    \begin{align*}
         \dim\,(A_{\alpha})_e &= q, & \dim\,(A_{\beta})_e &= r, &
          \dim\,(A)_e &= q + r,\\
         \dim\,(A_{\alpha})_{e+1} &= q^{\langle e \rangle}, & \dim\,(A_{\beta})_{e+1} &= r^{\langle e \rangle},&
          \dim\,(A)_{e+1} &= q^{\langle e \rangle} + r^{\langle e \rangle}.
    \end{align*}
    The last column follows from the observation that
      $A = A_{\alpha} \oplus A_{\beta}$, 
      while the bottom row in the first two columns is 
      a consequence of Lemma~\ref{lem_Macaulay_bound}.
    Since $\dim (A)_{e+1} \leqslant (q+r)^{\langle e \rangle}$ again by Lemma~\ref{lem_Macaulay_bound},
    we obtain
    \begin{equation}\label{equ_subadditivity_of_Macaulay_exponents}
      q^{\langle e \rangle} + r^{\langle e\rangle}\leqslant 
      (q+r)^{\langle e\rangle}.
    \end{equation}
    
    Next we show another special case of the claim with 
    $q= \binom{a}{d}$ for some $a$, $r=0$, and $e=d-1$.
    Then
    \begin{alignat*}{7}
      q &=&\tbinom{a}{d} &= \tbinom{a-1}{d-1}&+&\tbinom{a-2}{d-1} &+ 
      \dotsb  + & \tbinom{d}{d-1} &+& \tbinom{d-1}{d-1}, \text{ and}\\
      q^{\langle d \rangle} &= &\tbinom{a+1}{d+1} &=\hspace{0.9em} \tbinom{a}{d}&+&\tbinom{a-1}{d} &+ 
      \dotsb  +&\tbinom{d+1}{d} &+& \hspace{0.9em} \tbinom{d}{d}\\
        & & &  = \tbinom{a-1}{d-1}^{\langle d-1\rangle}&+&\tbinom{a-2}{d-1}^{\langle d-1\rangle} &+
      \dotsb  +&\tbinom{d}{d-1}^{\langle d-1\rangle} &+& \tbinom{d-1}{d-1}^{\langle d-1\rangle}\\
        & & &  \stackrel{\text{\eqref{equ_subadditivity_of_Macaulay_exponents}}}{\le} q^{\langle d-1\rangle}.
    \end{alignat*}
    Thus $\binom{a}{d}^{\langle d \rangle} \leqslant \binom{a}{d}^{\langle d -1\rangle}$.
    
    If $q$ is arbitrary, then we express it using the Macaulay coefficients:
    \[
       q= \tbinom{a_d}{d} + \underbrace{\tbinom{a_{d-1}}{d-1} + \dotsb + \tbinom{a_1}{1}}_{=:q'}
    \]
    with $a_d > a_{d-1} > \dotsb > a_1 \geqslant 0$.
    Then 
  \begin{equation}\label{equ_increasing_Macaulay_exponent_decreases_result}
    q^{\langle d \rangle} = 
    \tbinom{a_d}{d}^{\langle d \rangle} 
    + (q')^{\langle d -1 \rangle} 
    \leqslant \tbinom{a_d}{d}^{\langle d-1 \rangle} 
    + (q')^{\langle d -1 \rangle} \stackrel{\text{\eqref{equ_subadditivity_of_Macaulay_exponents}}}{\le} 
    \left( \tbinom{a_d}{d} + q'\right)^{\langle d -1 \rangle} = q^{\langle d -1 \rangle}.
  \end{equation}
  The final claim of the lemma follows 
     by successively applying
     \eqref{equ_increasing_Macaulay_exponent_decreases_result}
     to show $q^{\langle d \rangle} \leqslant q^{\langle e \rangle}$
     for any $e\leqslant d$ and then combining it 
     with \eqref{equ_subadditivity_of_Macaulay_exponents}.
\end{prf}

\begin{lemma}\label{lem_moving_ideals_with_few_entries}
  Suppose $d$, $e$, $q$, $r$, $n$ are non-negative integers
    such that $d \geqslant e > 0$, 
    $0 \leqslant q \leqslant \binom{n+d}{d}$, 
    $0 \leqslant r \leqslant \binom{n+e}{e}$,
    and $q+r \geqslant \binom{n+e}{e}$.
  Then 
  \[
    q^{\langle d \rangle} + r^{\langle e\rangle} 
  \leqslant \left(q+r - \tbinom{n+e}{e}\right)^{\langle d \rangle} + \tbinom{n+e}{e}^{\langle e \rangle}.
  \]
\end{lemma}
\begin{prf}
   We reinterpret the statement of the lemma in terms 
     of growth of ideals.
   Let $S = \CC[\fromto{\alpha_0,\alpha_1}{\alpha_n}]$, 
     and consider four ideals in $S$: 
     $I = (\Lex_d^q)$, $J = (\Lex_e^r)$, 
     $I'= \left(\Lex_d^{q+r - \binom{n+e}{e}}\right)$,
     $J'= 0$.
   The inequalities guarantee that the lengths of the above
     segments make sense (are non-negative and do not exceed 
     the dimensions of $S_d$ or $S_e$ respectively).
   The claim of the lemma is that
   \[
      \dim (S/I)_{d+1} + \dim (S/J)_{e+1}
      \leqslant  \dim (S/I')_{d+1} + \dim (S/J')_{e+1}.
   \]
   In terms of minimising the growth of ideals instead of maximising the growth of algebras 
   as in Remark~\ref{rem_growth_of_algebras_vs_growth_of_ideals},
   this is equivalent to
   \begin{equation}
     \label{equ_need_to_show_for_moving_all_monomials_to_one_part}
     \dim I_{d+1} + \dim J_{e+1} \geqslant \dim I'_{d+1} + 
     \dim J'_{e+1} = \dim I'_{d+1}.
   \end{equation}

   This last claim follows from Lemma~\ref{lem_Macaulay_bound}
     applied to 
     $K:=\alpha_0^{e+1} \cdot I_d + \alpha_1^{d+1} \cdot J_e \subset S_{d+e+1}$.
     Note that both $K$ and $K\cdot S_1$ are direct sums:
     \[
       K=\alpha_0^{e+1} \cdot I_d \oplus \alpha_1^{d+1} \cdot J_e \quad  \text{ and } \quad 
       K\cdot S_1 =\alpha_0^{e+1} \cdot I_{d+1} \oplus \alpha_1^{d+1} \cdot J_{e+1}.
     \]
     Thus $\dim K\cdot S_1$ is equal to the left hand side of
       \eqref{equ_need_to_show_for_moving_all_monomials_to_one_part}.
     The right hand side is equal to the dimension of  
       $\alpha_0^{e+1}\cdot I'_{d+1} = (\alpha_0^{e+1}\cdot I'_{d})\cdot S_1$,
       where $\alpha_0^{e+1}\cdot I'_{d}$ is   
       the lex-segment of the same dimension as $K$.
     Thus the inequality 
     \eqref{equ_need_to_show_for_moving_all_monomials_to_one_part}
     indeed follows from Lemma~\ref{lem_Macaulay_bound}.
\end{prf}

\subsection{Border rank of monomials in three variables and generalisations}

   We need a generalisation of Lemma~\ref{lem_Macaulay_bound} 
     to monomial ideals that are spread across a couple 
     of disjoint $S_{d_j}$.

   \begin{defin}\label{def_Lex_bar}
      Suppose $X=\PP^n$ and $S=\CC[\fromto{\alpha_0}{\alpha_n}]$.
      Let $\bar{d}=(\fromto{d_1}{d_j})$ 
         be a sequence of $j$ positive integers with 
         $d_1 \leqslant \dotsb \leqslant d_j$.
      Consider the linear space
         $\overline{S}_{\bar{d}} = S_{d_1} \oplus \dotsb \oplus S_{d_j}$.
         Let $\overline{S}_{\bar{d}, i}$ denote the $i$-th factor of $\overline{S}_{\bar{d}}$, isomorphic to $S_{d_i}$
           (note that there might be many factors isomorphic 
            to $S_{d_i}$).
      For an integer $0\leqslant r \leqslant \dim \overline{S}_{\bar{d}}$,
         define $\overline{\Lex}_{\bar{d}}^r$ 
         to be the linear subspace of $\overline{S}_{\bar{d}}$
         of codimension $r$
         such that for some $i_0\in \setfromto{1}{j}$
         we have:
         \begin{itemize}
          \item $\overline{\Lex}_{\bar{d}}^r$ is a direct sum 
                 of subspaces $\overline{\Lex}_{\bar{d},i}^r \subset \overline{S}_{\bar{d}, i}$,
          \item for all $i<i_0$, 
        $\overline{\Lex}_{\bar{d}, i}^r = 0$,
          \item for all $i>i_0$,  
        $\overline{\Lex}_{\bar{d}, i}^r = \overline{S}_{\bar{d}, i}$,
          \item $\overline{\Lex}_{\bar{d}, i_0}^r = \Lex_{d_{i_0}}^{r'}$ for some $r'$.
         \end{itemize}
      Informally,
         $\overline{\Lex}_{\bar{d}}^0 = \overline{S}_{\bar{d}}$,
         and to get $\overline{\Lex}_{\bar{d}}^{r}$ from 
         $\overline{\Lex}_{\bar{d}}^{r-1}$ we remove the 
         first monomial (in the \grevlex{} order) 
         from the first non-zero summand. 
   \end{defin}

   \begin{prop}\label{prop_Macaulay_bound_for_disjoint_modules}
      With  notation as in Definition~\ref{def_Lex_bar},
      suppose $W\subset \overline{S}_{\bar{d}}$ 
        is a direct sum  
        $\bigoplus_{i=1}^j W_i$ 
        of subspaces  $W_i \subset \overline{S}_{\bar{d},i}$.
      Let $r= \codim (W \subset \overline{S}_{\bar{d}})$ 
        and define $\bar{1} = (1,1,\dotsc, 1)$ ($j$ ones).
      Then 
      \[
         \codim\!\left(W \cdot S_1 \subset 
           \overline{S}_{\bar{d} + \bar{1}} \right)
         \leqslant \codim\!\left( \overline{\Lex}_{\bar{d}}^r \cdot S_1
         \subset \overline{S}_{\bar{d} + \bar{1}} \right).
      \]
  \end{prop}
  \begin{prf}
     If there is only one summand, that is, $j=1$, then the claim
       is Lemma~\ref{lem_Macaulay_bound}.
     If there are two summands, $j=2$,
     then the claim follows from 
     Lemmas~\ref{lem_moving_ideals_with_lots_of_entries} 
     and~\ref{lem_moving_ideals_with_few_entries}.
     
     Finally, suppose the number of summands $j$ is arbitrary.
     Pick the first summand $i$ such that 
        $W_i \ne 0$.
     If for all $i'>i$ we have $W_{i'} =\overline{S}_{\bar{d},i'}$, 
        then we are done by Lemma~\ref{lem_Macaulay_bound}.
     So suppose there is $i'>i$ such that $W_{i'} \ne \overline{S}_{\bar{d},i'}$.
     We use the ``two summands step'' 
        to move the dimension from $W_{i}$ to $W_{i'}$,
        arriving either at the case $W_{i'}=\overline{S}_{\bar{d},i'}$
        or $W_{i}=0$.
     Repeat the argument until there is at most one $i$
        such that $W_i \ne \overline{S}_{\bar{d},i}$
        and $W_i \ne 0$ and conclude the proof.
  \end{prf}
  The steps of the above proof are illustrated in 
     Example~\ref{ex_moving_ideal_around}, 
     while Example~\ref{ex_br_mono_4443_lower_bound} 
     shows how the proposition can be useful 
     to provide lower bounds on the border rank of monomials.

  \begin{example}\label{ex_moving_ideal_around}
     Suppose $n=3$, $S=\CC[\alpha_0,\alpha_1,\alpha_2,\alpha_3]$, and $\bar{d}= (2,3,3,4)$.
     The dimension of $\overline{S}_{\bar{d}} = S_2 \oplus S_3\oplus S_3 \oplus S_4$ is equal to 
        $85=10 + 20 + 20 + 35$.
     Suppose we are interested in a linear subspace
        $W= W_1\oplus W_2\oplus W_3\oplus W_4$ with 
        $W_1\subset S_2$, $W_2\subset S_3$, $W_3\subset S_3$, 
        and $W_4\subset S_4$ such that $\dim W = 47$ (equivalently,
        the codimension is $85-47=38$).
     We want to determine the minimal growth of the space to the next degree, that is, 
        we want $S_1\cdot W = S_1\cdot W_1 \oplus S_1 \cdot W_2\oplus S_1\cdot W_3\oplus S_1\cdot W_4$ 
        to have as small dimension as possible. 
     Equivalently, we want its codimension to be as large as possible.
            
     As an illustration, suppose for instance that $\dim W_1 =10 $, 
     $\dim W_2=17$, $\dim W_3=7$ and $\dim W_4=13$.
     Since we look for the minimal growth, we may assume that each $W_i$ is a lex-segment. 
     Then the proof of 
        Proposition~\ref{prop_Macaulay_bound_for_disjoint_modules}
        goes in the following steps:

     \medskip   
     \noindent 
     \begin{tabular}{|r|cccc|c|p{0.27\textwidth}|}
     \hline
         Step&\multicolumn{4}{r|}{Codimension of \qquad\qquad\quad\phantom{a}} & $S_1\cdot W \subset $ & action\\
         &$W_1\subset S_2$&$W_2\subset S_3$&$W_3\subset S_3$
         &$W_4\subset S_4$& $\overline{S}_{3,4,4,5}$&\\
       \hline  
       1. & $0$  & $3$  & $13$ & $22$ & $52$ & move $10$ from $W_1$ to $W_4$,\\
       2. & $10$ & $3$  & $13$ & $12$ & $57$ & move $12$ from $W_2$ to $W_4$,\\
       3. & $10$ & $15$ & $13$ & $0$  & $61$ & move $5$ from $W_2$ to $W_3$,\\
       4. & $10$ & $20$ & $8$  & $0$  & $\mathbf{65}$ & done!\\
       \hline
     \end{tabular}
     \medskip   

     That is, the maximal growth of the codimension is $65$.
     In the table above, in each row the codimension of $S_1\cdot W \subset \overline{S}_{3,4,4,5}$ is obtained as a sum of Macaulay exponents.
     For instance, in the second step, 
     $
      57 = 10^{\langle 2\rangle} +3^{\langle 3\rangle}+13^{\langle 3\rangle}+12^{\langle 4\rangle} = 20+ 3+ 19 + 15
     $.
  \end{example}
   
   \begin{example}\label{ex_br_mono_4443_lower_bound}
      Let $X=\PP^3$
         and $S=\CC[\alpha_0, \alpha_1, \alpha_2, \alpha_3]$ 
         be the Cox ring of $X$.
      Consider the monomial 
         $F=x_0^{(4)}x_1^{(4)}x_2^{(4)}x_3^{(3)}$.
      It is known that $70 \leqslant \borderrank[\PP^3]{F}\leqslant 100$:
      the lower bound is the catalecticant bound
      \cite[Thm~11.2]{landsberg_teitler_ranks_and_border_ranks_of_symm_tensors} 
      and the upper bound is \eqref{equ_upper_bound_monos}.
      Using Proposition~\ref{prop_Macaulay_bound_for_disjoint_modules}, 
      we can show that $\borderrank[\PP^3]{F} \geqslant 86$, halving the range of possible border rank values.
      Indeed, consider $\apolar_8$ and $\apolar_9$. 
       Note that 
       \[
          \apolar_8 = \alpha_0^5\cdot S_3\oplus\alpha_1^5\cdot S_3
          \oplus\alpha_2^5\cdot S_3\oplus\alpha_3^4\cdot S_4 \simeq \overline{S}_{(3,3,3,4)}
       \]
       and $\dim\apolar_8 = 3\cdot 20 + 35 = 95$, while $\dim S_8 =165$.
       Similarly, $\dim\apolar_9=158= 3\cdot 35 + 56 -3$
       (the $-3$ is due to three syzygies between generators of $\apolar$ in degree~$9$:
       $\alpha_i^5 \cdot \alpha_3^4$ can be obtained in two ways for each $i\in \set{0,1,2}$) and  $\dim S_9 =220$.
       
       Suppose on the contrary that $\borderrank{F}\leqslant 85$.
       Then by move-fit (Corollary~\ref{cor_monomials})
         there is a monomial ideal $I\subset \apolar$
         with $\dim\ (S/I)_8=\dim\ (S/I)_9=85$.
       Therefore $\codim\!\left(I_8 \subset \apolar_8\right) = 15$ and  $\codim\!\left(I_9 \subset \apolar_9\right) = 23$.
       Set $W$ to be  $I_8$ viewed as a subspace of $\overline{S}_{(3,3,3,4)}$. 
       Note that since $I$ is a monomial ideal, the space $W$ satisfies
       the assumption of Proposition~\ref{prop_Macaulay_bound_for_disjoint_modules},
       $W=W_1\oplus W_2 \oplus W_3 \oplus W_4$ 
         with $W_i \subset \overline{S}_{(3,3,3,4), i}$.
       By Proposition~\ref{prop_Macaulay_bound_for_disjoint_modules}
         we have $\codim\!\left(S_1\cdot W \subset \overline{S}_{(4,4,4,5)}\right)
          \leqslant 15^{\langle 3 \rangle} = 22$.
       The map 
       \[
         \overline{S}_{(4,4,4,5)} \to \apolar_9 = \alpha_0^5\cdot S_4 +\alpha_1^5\cdot S_4
          +\alpha_2^5\cdot S_4+\alpha_3^4\cdot S_5
       \]
       is surjective, thus the image of $S_1\cdot W$ under this map also has codimension at most $22$ in $\apolar_9$.
       Since the image of $S_1 \cdot W$ is contained in $I_9$,
         it follows that 
       \[
         \codim \left(I_9\subset \apolar_9\right)\leqslant 22,
       \]
       a contradiction.
       This concludes the proof  that $\borderrank[\PP^3]{x_0^{(4)}x_1^{(4)}x_2^{(4)}x_3^{(3)}} \geqslant 86$.
   \end{example}

  \begin{thm}\label{thm_monomials}
     Suppose $X= \PP^2\times Y$ with a smooth toric projective variety $Y$, and 
     $\widetilde{S}[X] = \CC[x_0, x_1, x_2, \  y_1, \dotsc, y_{n+w-3}]$ 
       with the first three variables corresponding to $\PP^2$
       and the remaining variables corresponding to $Y$.
     Fix two monomials,
     $F= x_0^{(a_0)}x_1^{(a_1)}x_2^{(a_2)}$ 
       with $a_0\geqslant a_1 \geqslant a_2$ 
       and $G\in \widetilde{S}[Y] = \CC[y_1, \dotsc, y_{n+w-3}]$. 
     Fix also a degree $D\in \Pic(Y)$ such that $\apolar[G]_D = 0$.
     Then 
     \[
       \borderrank{F\cdot G} \geqslant (a_1+1) \cdot (a_2+1) \cdot \dim H^0(Y, D).
     \]
 \end{thm}
   
 \begin{prf}
    Suppose by way of contradiction that 
    \[
      \borderrank{F\cdot G} \leqslant r:= (a_1+1) \cdot (a_2+1) 
         \cdot \dim H^0(Y, D) - 1.
    \]
    We will be considering the degree 
      $(a_1+a_2, D) \in \Pic(X) = \Pic(\PP^2 \times Y)
      =\ZZ\oplus \Pic(Y)$.
    By Corollary~\ref{cor_monomials},
    there exists a linear subspace 
    $W \subset \apolar[F\cdot G]_{(a_1+a_2, D)}$ 
    spanned by monomials
        such that
    \begin{equation}\label{equ_need_to_show_for_monomials}
        \codim \left(W\subset S[X]_{(a_1+a_2, D)}\right) = r 
        \text{ and }
        \codim \left( W \cdot S[\PP^2]_1
              \subset S[X]_{(a_1+a_2+1, D)}\right) \geqslant r.
    \end{equation}
    We will show this is impossible. Note that 
    \begin{align*}
      \apolar[F\cdot G]_{(a_1+a_2, D)} =&
      \ \  \ \
      \alpha_0^{a_0+1} \cdot S[\PP^2]_{a_1+a_2-a_0-1} \otimes H^0(Y,D)\\
      &\oplus 
      \alpha_1^{a_1+1} \cdot S[\PP^2]_{a_2-1} \otimes H^0(Y,D)\\
      &\oplus
      \alpha_2^{a_2+1} \cdot S[\PP^2]_{a_1-1} \otimes H^0(Y,D)
    \end{align*}
    and an analogous equality holds for 
    $\apolar[F\cdot G]_{(a_1+a_2+1, D)}$. 
    (If $a_1+a_2-a_0-1 < 0$, then we skip it in the above formula
    and we also skip the corresponding parts in the proof below.)
    Thus we are in the situation of 
    Proposition~\ref{prop_Macaulay_bound_for_disjoint_modules} 
    for $n=2$, $j=3\dim H^0(Y,D)$ and 
    \[
      \bar{d}= (\underbrace{a_1+a_2-a_0-1,\dotsc,a_1+a_2-a_0-1}_{\dim H^0(Y,D) \text{ times}},
      \underbrace{a_2-1,\dotsc,a_2-1}_{\dim H^0(Y,D) \text{ times}},
      \underbrace{a_1-1,\dotsc,a_1-1}_{\dim H^0(Y,D) \text{ times}}).
    \]
    Thus 
    $\codim \left( W \cdot S[\PP^2]_1
              \subset \apolar_{(a_1+a_2+1, D)}\right) \leqslant 
     \codim \left(\overline{\Lex}_{\bar{d}}^{r'} \cdot S_1
         \subset \overline{S}_{\bar{d} + \bar{1}} \right)$,
      where 
      \[
         r'= r- \codim \left(\apolar_{(a_1+a_2, D)}\! \subset \! S[X]_{(a_1+a_2, D)}\right) = 
         \dim \left(S[\PP^2]_{a_1+a_2-a_0-1} \otimes H^0(Y,D) \right)-1.
      \]
    The explicit values of $r$ and $r'$ are such that 
       $\overline{\Lex}_{\bar{d}}^{r'}$ contains 
       all  the $S[\PP^2]_{a_1-1}$ summands,
       all  the $S[\PP^2]_{a_2-1}$ summands,
       and in addition one more monomial from one of the $S[\PP^2]_{a_1+a_2-a_0-1}$ summands.
    Thus
    \begin{align*}
       &\codim \left( W \cdot S[\PP^2]_1
              \subset S[X]_{(a_1+a_2+1, D)}\right)
       \\ &=
       \codim \left( W \cdot S[\PP^2]_1
              \subset \apolar_{(a_1+a_2+1, D)}\right)\\
       &\phantom{= } \ +      \codim \left(\apolar_{(a_1+a_2+1, D)}
              \subset  S[X]_{(a_1+a_2+1, D)}\right)
       \\ &\leqslant
            \codim \left(\overline{\Lex}_{\bar{d}}^{r'} \cdot S_1
         \subset \overline{S}_{\bar{d} + \bar{1}} \right)
         +       \codim \left(\apolar_{(a_1+a_2+1, D)}
              \subset  S[X]_{(a_1+a_2+1, D)}\right)
       \\& = r-2,
    \end{align*}
    contradicting \eqref{equ_need_to_show_for_monomials}.    
\end{prf}

In the following examples we keep the notation of 
  Theorem~\ref{thm_monomials}.
\begin{example}\label{ex_monomials_P2}
   If $Y = \set{\ast}$, then Theorem~\ref{thm_monomials} 
     proves that  the border rank of monomials in three variables 
     is equal to the upper bound of \eqref{equ_upper_bound_monos}.
\end{example}

\begin{example}\label{ex_monomials_P2xP1xP1___}
   Suppose $Y = \PP^1\times \PP^1 \times \dotsb$ is a finite product of projective lines.
   Take a monomial 
   $F = x_0^{(a_0)}x_1^{(a_1)}x_2^{(a_2)}\cdot
        y_0^{(b_0)}y_1^{(b_1)}\cdot
        z_0^{(c_0)}z_1^{(c_1)}\dotsm$ such that 
        $a_0\geqslant a_1\geqslant a_2$,
        $b_0\geqslant b_1$,
        $c_0\geqslant c_1$,\dots 
   Then 
   \[
     \borderrank[\PP^2\times \PP^1\times \PP^1 \times \dots]{F}= 
      (a_1 +1)(a_2+1)(b_1+1)(c_1+1)\dotsm.
   \]
   Indeed, use Theorem~\ref{thm_monomials}
      with $D=(b_1,c_1,\dotsc)$ to get the lower bound, 
      and Lemma~\ref{lem_monomial_upper_bound_on_general_toric_variety}
      to obtain the upper bound.
\end{example}

\section{Further research}\label{sec_further_research}

In this section we briefly discuss further research plans related to border apolarity.

\subsection{Efficiency}\label{sec_efficiency}

The method of border apolarity has already shown its potential 
  for new lower bounds for border rank.
However, it needs to be determined how much further we can work with this method.
Originally, it was intended to uniformly describe ``wild'' cases, 
where the border rank is less than the smoothable rank, as briefly discussed in Sections~\ref{sec_tensor_of_border_rank_3} 
  and \ref{sec_wild_cubic}, and in more detail 
  in \cite{nisiabu_jabu_smoothable_rank_example}.
  
Many (or all) classical and modern criteria for the border rank are cursed with
the ``cactus barriers'', that is, they cannot provide bounds for the  border rank
beyond the border cactus rank,
  which (for $X$ of large dimension) 
  is much lower than the border rank. 
Thus the natural question is whether the present method
  is also subject to these cactus barriers.
Seemingly, the immediate answer is negative,
  as the statement of Theorem~\ref{thm_nonsaturated_apolarity} 
  is ``if and only if''.
However, there is a catch, as at this moment we do not have enough criteria to determine if an ideal $I\in \Hilb_{S}^{h_r}$ is contained in $\Slip_{r, X}$ or not. 
Analogously, in the applications discussed in 
   Sections~\ref{sec_applications}, \ref{sec_monomials} and \cite{conner_harper_landsberg_border_apolarity_I} 
   only one implication of Theorem~\ref{thm_nonsaturated_apolarity} is used:
   if there is no ideal $I\in \Hilb_{S}^{h_r}$ 
   such that $I \subset \apolar$, then $\borderrank{F} >r$.
With this simplified approach, 
   the cactus curse is still (partially) valid,
   as we discuss in Subsection~\ref{sec_other_ranks}.
However, in Subsection~\ref{sec_smoothable_ideals} 
   we also briefly report on research in progress
   on conditions for ideals to be 
   in $\Slip_{r, X}$ or not.
It is a subject of ongoing joint work with Landsberg and his research group
  to determine if these methods combined with the study of smoothability 
  of finite schemes can beat the curse of cactus barriers. 

\subsection{Other variants of rank}\label{sec_other_ranks}

In the standard notation 
(as in Section~\ref{sec_multihomogeneous_coordinates}) 
for $X$ and $L$, if $W\subset H^0(X,L)^*$ is a linear subspace,
   then the \emph{rank} of $W$ is the minimal integer $r$
   such that $\PP(W)\subset \linspan{\fromto{p_1}{p_r}}$ 
   for some $p_i\in X\subset \PPof{H^0(X,L)^*}$.
This is sometimes referred to as the  \emph{simultaneous rank} of $W$.
It also has its border analogue: $W$ is of border rank at most $r$ 
   if and only if it is a limit of linear subspaces
   (of the same dimension) which have rank $r$. 
See for instance \cite{landsberg_jabu_ranks_of_tensors},
\cite{jabu_postinghel_rupniewski_Strassen_for_small_tensors}
for more details about these notions. 

It is straightforward to generalise apolarity 
  (Proposition~\ref{prop_multigraded_apolarity})
  and border apolarity 
  (Theorem~\ref{thm_nonsaturated_apolarity})
  to the simultaneous case, and we discuss it in detail in 
  a forthcoming paper.
  
Other variants of rank are the \emph{cactus rank} 
   and \emph{border cactus rank}:
   see \cite{ranestad_schreyer_on_the_rank_of_a_symmetric_form}, 
       \cite{nisiabu_jabu_cactus},
       \cite{nisiabu_jabu_smoothable_rank_example},
       \cite{bernardi_brachat_mourrain_comparison}, 
       \cite{galazka_vb_cactus}, and other related work.
Roughly, the cactus rank arises from linear spans of 
   finite subschemes of $X$ that are not necessarily smooth
   (that is, not necessarily equal to the disjoint union of 
    reduced points).
It is known (see the references above) 
   that most determinantal lower bounds for border rank 
   are in fact lower bounds for border cactus rank, and the latter tends to be much lower than the former.
Thus it is desirable to construct bounds for border rank
   that do not apply for cactus rank.

Our bound arising from border apolarity can also be modified to similar statements for border cactus rank 
 (and also simultaneous border cactus rank).
As a consequence, without taking into account the discussion
  of  membership in $\Slip_{r,X}$,
  it might be hard to use border apolarity 
  to distinguish between border rank and border cactus rank.
Yet the arguments leading to the generalisation
   and even the formulation of the border-cactus-apolarity are not so 
   straightforward 
   and they require using the techniques announced 
   in Subsection~\ref{sec_saturation_open}.
Again, the details will be provided in a forthcoming work.
       
\subsection{Limit ideals of points}\label{sec_smoothable_ideals}

There are four types of irreducible 
   components in
   the multigraded Hilbert scheme $\Hilb_S^{h_r}$.
Two of the types have as a general member a saturated ideal,
   while the remaining two types only consist of non-saturated ideals.
In the other direction, we ask if the scheme defined by
   the general ideal is reduced or not.
Combining these two properties we get our four types, 
   including one type whose general ideal is saturated and defines a reduced subscheme.
   This type consists of a unique component, 
      namely $\Slip_{r,X}$.

Given the discussion in 
Subsections~\ref{sec_efficiency} and \ref{sec_other_ranks}
  and the relations of the four types to border apolarity and border-cactus-apolarity,
it seems critical to learn how to distinguish the four types.
In particular, for $I\in \Hilb_{S}^{h_r}$ we have to decide:
\begin{itemize}
 \item Is $I$ a limit of ideals whose saturation is radical?
 \item Is $I$ a limit of saturated ideals?  
\end{itemize}
If $I\in \Slip_{r,X}$, then certainly the answer to both above questions is affirmative. 
However, at this moment, we do not know if the converse holds.
The first item is an intensively studied topic of 
   \emph{smoothability of finite schemes}; 
   see for instance \cite{cartwright_erman_velasco_viray_Hilb8},
   \cite{casnati_jelisiejew_notari_Hilbert_schemes_via_ray_families},
   \cite{jelisiejew_PhD},
   \cite{jabu_jeliesiejew_finite_schemes_and_secants}, and
   \cite{douvropoulos_jelisiejew_nodland_teitler_11_points_on_A3}.
The second question is new and research in this direction 
   is being conducted by Ma{\'n}dziuk.

\subsection{Open locus of points in general position}\label{sec_saturation_open}

In Propositions~\ref{prop_saturated_is_open},
   \ref{prop_motivation_for_h_r_X}, and also in
   Lemma~\ref{lem_hilb_function_of_very_general_tuple},
   we show that appropriate subsets are dense or empty. 
A stronger claim is in fact true: these subsets 
   (for instance $\Sip_{r,X} \subset \Slip_{r,X}$) 
   are all Zariski open. 
This is fairly standard when $X=\PP^n$, 
   but in order to show it in general 
   (for any smooth projective toric variety $X$),
   we need an algebraic statement,
   that the set of saturated fibres in a flat family of ideals 
   is Zariski open.
A proof of this claim was communicated to us by 
   Jelisiejew and it is both a very interesting observation 
   on its own,
   and also relevant to the proofs of claims
   about border cactus rank 
   and the corresponding apolarity theory.

Again, this will be detailed in a forthcoming work.   

\subsection{Other base fields}\label{sec_other_base_fields}

We also claim that the results of our article can be extended
   to any other algebraically closed base field $\kk$.
There are two issues that should be resolved.

The first problem is the lack of solid reference for 
  Cox rings of toric varieties over base fields $\kk\ne \CC$.
The general consensus among experts is that ``everything works fine'' 
  and ``toric varieties are defined and can be studied without 
  any problem over $\Spec \ZZ$''.
However, this claim is not properly documented and 
  some subtle issues may appear. 
For instance, whenever the class group has torsion 
  of the same order as the characteristic of the base field, 
  quotients by non-reduced group schemes are necessary to consider
  the quotient construction \cite[\S5.1]{cox_book}.
But this will never happen in the setting of smooth projective 
  toric varieties, and in the special cases of interest, 
  namely Segre--Veronese varieties, the ideal--subscheme 
  correspondence is straightforward and clear.
Moreover, more general cases
  are treated 
  in~\cite[Thm~(4.4.3)]{rohrer_qcoh_sheaves_on_toric_schemes}.
  
  Another issue
  is the ``very general'' property.
Over the complex numbers this is a dense property, 
  but over countable fields $\kk$ it may very well mean 
  just an empty set.
This is again dealt with by the Zariski openness of
  the set of saturated fibres
  (mentioned in Subsection~\ref{sec_saturation_open}).
  
Details of this approach will be provided in our forthcoming work.

\subsection{Varieties of sums of powers}

In the notation of Section~\ref{sec_VSP},
   looking at $\bVSP(F,r)$ might help 
   to (partially) resolve the singularities appearing 
   in $\mathit{VSP}(F,r)$.
   Since $\bVSP(F,r)$ is defined in a natural way, 
     it may be easier to study in cases of interest.
   In fact, in the course of proofs about $\mathit{VSP}(F,r)$
     existing so far, one of the key technical steps
     implies the claim 
     (implicitly, as the terminology was not present at that time)
     that $\mathit{VSP}(F,r)$ is equal to $\bVSP(F,r)$:
     see for instance 
     \cite[Prop.~3.1]{ranestad_voisin_VSP_and_divisors_in_the_moduli_of_cubic_fourfolds}.
     
It is an interesting open problem 
  to investigate 
  the details of the interaction between the two varieties $\mathit{VSP}(F,r)$ and $\bVSP(F,r)$.
In particular, this question includes 
  understanding the conditions on $F$, 
  and also on $X$ and $L$, 
  where $F\in \PPof{H^0(X, L)^*}$,
  that force the two varieties to be isomorphic.

\bibliography{apolarity-for-border-rank}

\def\polhk#1{\setbox0=\hbox{#1}%
  {\ooalign{\hidewidth\lower1.5ex\hbox{`}\hidewidth\crcr\unhbox0}}}\def\dbar{\leavevmode\hbox
  to 0pt{\hskip.2ex\accent"16\hss}d}
\begin{thebibliography}{BBKT15}

\bibitem[BB13a]{ballico_bernardi_4th_secant_of_Veronese}
Edoardo Ballico and Alessandra Bernardi.
\newblock Stratification of the fourth secant variety of {V}eronese varieties
  via the symmetric rank.
\newblock {\em Adv. Pure Appl. Math.}, 4(2):215--250, 2013.

\bibitem[BB13b]{brown_jabu_maps_of_toric_varieties}
Gavin Brown and Jaros\l{}aw Buczy\'nski.
\newblock Maps of toric varieties in {C}ox coordinates.
\newblock {\em Fund. Math.}, 222:213--267, 2013.

\bibitem[BB14]{nisiabu_jabu_cactus}
Weronika Buczy{\'n}ska and Jaros{\l}aw Buczy{\'n}ski.
\newblock Secant varieties to high degree {V}eronese reembeddings,
  catalecticant matrices and smoothable {G}orenstein schemes.
\newblock {\em J. Algebraic Geom.}, 23:63--90, 2014.

\bibitem[BB15]{nisiabu_jabu_smoothable_rank_example}
Weronika Buczy{\'n}ska and Jaros{\l}aw Buczy{\'n}ski.
\newblock On differences between the border rank and the smoothable rank of a
  polynomial.
\newblock {\em Glasg. Math. J.}, 57(2):401--413, 2015.

\bibitem[BBKT15]{nisiabu_jabu_kleppe_teitler_direct_sums}
Weronika Buczy\'nska, Jaros{\l}aw Buczy\'nski, Johannes Kleppe, and Zach
  Teitler.
\newblock Apolarity and direct sum decomposability of polynomials.
\newblock {\em Michigan Math. J.}, 64(4):675--719, 2015.

\bibitem[BBM14]{bernardi_brachat_mourrain_comparison}
Alessandra Bernardi, J\'{e}r\^{o}me Brachat, and Bernard Mourrain.
\newblock A comparison of different notions of ranks of symmetric tensors.
\newblock {\em Linear Algebra Appl.}, 460:205--230, 2014.

\bibitem[BBT13]{nisiabu_jabu_teitler_Waring_decompositions_of_monomials}
Weronika Buczy{\'n}ska, Jaros{\l}aw Buczy{\'n}ski, and Zach Teitler.
\newblock Waring decompositions of monomials.
\newblock {\em J. Algebra}, 378:45--57, 2013.

\bibitem[BCP97]{magma}
Wieb Bosma, John Cannon, and Catherine Playoust.
\newblock The {M}agma algebra system. {I}. {T}he user language.
\newblock {\em J. Symbolic Comput.}, 24(3-4):235--265, 1997.
\newblock Computational algebra and number theory (London, 1993). Available for
  use on-line at \url{http://magma.maths.usyd.edu.au/calc/}.

\bibitem[BGI11]{bernardi_gimigliano_ida}
Alessandra Bernardi, Alessandro Gimigliano, and Monica Id{\`a}.
\newblock Computing symmetric rank for symmetric tensors.
\newblock {\em J. Symbolic Comput.}, 46(1):34--53, 2011.

\bibitem[BH93]{bruns_herzog_Cohen_Macaulay_rings}
Winfried Bruns and J\"{u}rgen Herzog.
\newblock {\em Cohen-{M}acaulay rings}, volume~39 of {\em Cambridge Studies in
  Advanced Mathematics}.
\newblock Cambridge University Press, Cambridge, 1993.

\bibitem[BJ17]{jabu_jeliesiejew_finite_schemes_and_secants}
Jaros{\l}aw Buczy\'{n}ski and Joachim Jelisiejew.
\newblock Finite schemes and secant varieties over arbitrary characteristic.
\newblock {\em Differential Geom. Appl.}, 55:13--67, 2017.

\bibitem[BL13]{landsberg_jabu_ranks_of_tensors}
Jaros{\l}aw Buczy{\'n}ski and Joseph~M. Landsberg.
\newblock Ranks of tensors and a generalization of secant varieties.
\newblock {\em Linear Algebra Appl.}, 438(2):668--689, 2013.

\bibitem[BL14]{landsberg_jabu_third_secant}
Jaros\l~aw Buczy\'{n}ski and J.~M. Landsberg.
\newblock On the third secant variety.
\newblock {\em J. Algebraic Combin.}, 40(2):475--502, 2014.

\bibitem[Bore91]{borel}
Armand Borel.
\newblock {\em Linear algebraic groups}, volume 126 of {\em Graduate Texts in
  Mathematics}.
\newblock Springer-Verlag, New York, second edition, 1991.

\bibitem[BPR20]{jabu_postinghel_rupniewski_Strassen_for_small_tensors}
Jaros{\l}aw Buczy\'{n}ski, Elisa Postinghel, and Filip Rupniewski.
\newblock On {S}trassen's rank additivity for small three-way tensors.
\newblock {\em SIAM J. Matrix Anal. Appl.}, 41(1):106--133, 2020.

\bibitem[CCG12]{carlini_catalisano_geramita_Waring_rank_of_monomials}
Enrico Carlini, Maria~Virginia Catalisano, and Anthony~V. Geramita.
\newblock The solution to the {W}aring problem for monomials and the sum of
  coprime monomials.
\newblock {\em J. Algebra}, 370:5--14, 2012.

\bibitem[CEVV09]{cartwright_erman_velasco_viray_Hilb8}
Dustin~A. Cartwright, Daniel Erman, Mauricio Velasco, and Bianca Viray.
\newblock Hilbert schemes of 8 points.
\newblock {\em Algebra Number Theory}, 3(7):763--795, 2009.

\bibitem[Chia04]{chiantini_lectures_on_structure_of_projective_embeddings}
L.~Chiantini.
\newblock Lectures on the structure of projective embeddings.
\newblock {\em Rend. Sem. Mat. Univ. Politec. Torino}, 62(4):335--388, 2004.

\bibitem[CHL19]{conner_harper_landsberg_border_apolarity_I}
Austin Conner, Alicia Harper, and Joseph~M. Landsberg.
\newblock Border apolarity of tensors {I}: new lower bounds for matrix
  multiplication and $\operatorname{det}_3$.
\newblock arXiv:1911.07981, 2019.

\bibitem[CJN15]{casnati_jelisiejew_notari_Hilbert_schemes_via_ray_families}
Gianfranco Casnati, Joachim Jelisiejew, and Roberto Notari.
\newblock Irreducibility of the {G}orenstein loci of {H}ilbert schemes via ray
  families.
\newblock {\em Algebra Number Theory}, 9(7):1525--1570, 2015.

\bibitem[CLS11]{cox_book}
David~A. Cox, John~B. Little, and Henry~K. Schenck.
\newblock {\em Toric varieties}, volume 124 of {\em Graduate Studies in
  Mathematics}.
\newblock American Mathematical Society, Providence, RI, 2011.

\bibitem[Cox95]{cox_homogeneous}
David~A. Cox.
\newblock The homogeneous coordinate ring of a toric variety.
\newblock {\em J. Algebraic Geom.}, 4(1):17--50, 1995.

\bibitem[DJNT17]{douvropoulos_jelisiejew_nodland_teitler_11_points_on_A3}
Theodosios Douvropoulos, Joachim Jelisiejew, Bernt Ivar~Utst{\o}l N{\o}dland,
  and Zach Teitler.
\newblock The {H}ilbert scheme of 11 points in {$\mathbb{A}^3$} is irreducible.
\newblock In {\em Combinatorial algebraic geometry}, volume~80 of {\em Fields
  Inst. Commun.}, pages 321--352. Fields Inst. Res. Math. Sci., Toronto, ON,
  2017.

\bibitem[DT15]{derksen_teitler_lower_bounds_for_ranks_of_invariant_forms}
Harm Derksen and Zach Teitler.
\newblock Lower bound for ranks of invariant forms.
\newblock {\em J. Pure Appl. Algebra}, 219(12):5429--5441, 2015.

\bibitem[Frie13]{friedland_salmon_problem_paper}
Shmuel Friedland.
\newblock On tensors of border rank {$l$} in {$\mathbb{C}^{m\times n\times
  l}$}.
\newblock {\em Linear Algebra Appl.}, 438(2):713--737, 2013.

\bibitem[Ga{\l}{\k a}14]{galazka_mgr}
Maciej Ga{\l}{\k a}zka.
\newblock Multigraded apolarity.
\newblock Master's thesis, Institute of Mathematics, Warsaw University, 2014.
\newblock
  \url{https://www.mimuw.edu.pl/~jabu/teaching/Theses/MSc_thesis_Galazka.pdf},
  arxiv:1601.06211.

\bibitem[Ga{\l}{\k a}17]{galazka_vb_cactus}
Maciej Ga{\l}{\k a}zka.
\newblock Vector bundles give equations of cactus varieties.
\newblock {\em Linear Algebra Appl.}, 521:254--262, 2017.

\bibitem[Gree98]{green_generic_initial_ideals}
Mark~L. Green.
\newblock Generic initial ideals.
\newblock In {\em Six lectures on commutative algebra}, volume 166 of {\em
  Progress in Mathematics}, pages 119--186. Birkh\"auser Verlag, Basel, 1998.

\bibitem[GRV18]{gallet_ranestad_villamizar}
Matteo Gallet, Kristian Ranestad, and Nelly Villamizar.
\newblock Varieties of apolar subschemes of toric surfaces.
\newblock {\em Ark. Mat.}, 56(1):73--99, 2018.

\bibitem[Harr95]{harris}
Joe Harris.
\newblock {\em Algebraic geometry}, volume 133 of {\em Graduate Texts in
  Mathematics}.
\newblock Springer-Verlag, New York, 1995.
\newblock A first course, Corrected reprint of the 1992 original.

\bibitem[Hart77]{hartshorne}
Robin Hartshorne.
\newblock {\em Algebraic geometry}.
\newblock Springer-Verlag, New York, 1977.
\newblock Graduate Texts in Mathematics, No. 52.

\bibitem[HMV19]{huang_michalek_ventura_wild_forms}
Hang Huang, Mateusz Micha{\l}ek, and Emanuele Ventura.
\newblock Vanishing {H}essian, wild forms and their border {VSP}.
\newblock arXiv:1912.13174, 2019.

\bibitem[HS04]{haiman_sturmfels_multigraded_Hilb}
Mark Haiman and Bernd Sturmfels.
\newblock Multigraded {H}ilbert schemes.
\newblock {\em J. Algebraic Geom.}, 13(4):725--769, 2004.

\bibitem[IK99]{iarrobino_kanev_book_Gorenstein_algebras}
Anthony Iarrobino and Vassil Kanev.
\newblock {\em Power sums, {G}orenstein algebras, and determinantal loci},
  volume 1721 of {\em Lecture Notes in Mathematics}.
\newblock Springer-Verlag, Berlin, 1999.
\newblock Appendix C by Iarrobino and Steven L. Kleiman.

\bibitem[IR01]{iliev_ranestad_K3_of_genus_8_and_VSP}
Atanas Iliev and Kristian Ranestad.
\newblock {$K3$} surfaces of genus 8 and varieties of sums of powers of cubic
  fourfolds.
\newblock {\em Trans. Amer. Math. Soc.}, 353(4):1455--1468, 2001.

\bibitem[Jeli17]{jelisiejew_PhD}
Joachim Jelisiejew.
\newblock {\em Hilbert schemes of points and their applications}.
\newblock PhD thesis, University of Warsaw, 2017.
\newblock
  \url{https://www.mimuw.edu.pl/~jabu/teaching/Theses/PhD_Thesis_Jelisiejew.pdf}.

\bibitem[Land17]{landsberg_geometry_and_complexity}
Joseph~M. Landsberg.
\newblock {\em Geometry and complexity theory}, volume 169 of {\em Cambridge
  Studies in Advanced Mathematics}.
\newblock Cambridge University Press, Cambridge, 2017.

\bibitem[LM17]{landsberg_michalek_geometry_of_border_rank_decompositions}
Joseph~M. Landsberg and Mateusz Micha{\l}ek.
\newblock On the geometry of border rank decompositions for matrix
  multiplication and other tensors with symmetry.
\newblock {\em SIAM J. Appl. Algebra Geom.}, 1(1):2--19, 2017.

\bibitem[LT10]{landsberg_teitler_ranks_and_border_ranks_of_symm_tensors}
Joseph~M. Landsberg and Zach Teitler.
\newblock On the ranks and border ranks of symmetric tensors.
\newblock {\em Found. Comput. Math.}, 10(3):339--366, 2010.

\bibitem[Oedi19]{oeding_border_rank_monomials}
Luke Oeding.
\newblock Border ranks of monomials.
\newblock arXiv:1608.02530, 2019.

\bibitem[Rohr14]{rohrer_qcoh_sheaves_on_toric_schemes}
Fred Rohrer.
\newblock Quasicoherent sheaves on toric schemes.
\newblock {\em Expo. Math.}, 32(1):33--78, 2014.

\bibitem[RS00]{ranestad_schreyer_VSP}
Kristian Ranestad and Frank-Olaf Schreyer.
\newblock Varieties of sums of powers.
\newblock {\em J. Reine Angew. Math.}, 525:147--181, 2000.

\bibitem[RS11]{ranestad_schreyer_on_the_rank_of_a_symmetric_form}
Kristian Ranestad and Frank-Olaf Schreyer.
\newblock On the rank of a symmetric form.
\newblock {\em J. Algebra}, 346:340--342, 2011.

\bibitem[Russ16]{russo_geometry_of_special_varieties}
Francesco Russo.
\newblock {\em On the geometry of some special projective varieties}, volume~18
  of {\em Lecture Notes of the Unione Matematica Italiana}.
\newblock Springer, Cham; Unione Matematica Italiana, Bologna, 2016.

\bibitem[RV17]{ranestad_voisin_VSP_and_divisors_in_the_moduli_of_cubic_fourfolds}
Kristian Ranestad and Claire Voisin.
\newblock Variety of power sums and divisors in the moduli space of cubic
  fourfolds.
\newblock {\em Doc. Math.}, 22:455--504, 2017.

\bibitem[Teit14]{teitler_lower_bound_for_generalized_ranks}
Zach Teitler.
\newblock Geometric lower bounds for generalized ranks.
\newblock arXiv: 1406.5145, 2014.

\bibitem[Zak93]{zak_tangents}
F.~L. Zak.
\newblock {\em Tangents and secants of algebraic varieties}, volume 127 of {\em
  Translations of Mathematical Monographs}.
\newblock American Mathematical Society, Providence, RI, 1993.
\newblock Translated from the Russian manuscript by the author.

\end{thebibliography}
\bibliographystyle{alpha_four}

\end{document}